\documentclass[11pt,a4paper,reqno]{amsart}
\usepackage{vmargin,color}
\usepackage[latin1]{inputenc}
\usepackage{amsmath,amsfonts,amsthm,epsfig,graphicx}
\usepackage{stmaryrd,amssymb,mathrsfs} 
\usepackage{caption}

\newtheorem{theorem}{Theorem}[section]
\newtheorem*{theorem*}{Theorem}
\newtheorem{remark}{Remark}[section]
\newtheorem{definition}[remark]{Definition}

\newtheorem{proposition}[remark]{Proposition}

\newtheorem{example}[remark]{Example}

\numberwithin{equation}{section}
\numberwithin{figure}{section}

\newcommand{\Na}{\mathbb{N}} 
\newcommand{\Z}{\mathbb{Z}} 
\newcommand{\R}{\mathbb{R}} 
\newcommand{\Sf}{\mathbb{S}} 

\newcommand{\Ha}{\mathcal{H}} 

\newcommand{\D}{\mathscr{D}}

\newcommand{\e}{\varepsilon} 
\newcommand{\spt}{\mathrm{spt}} 
\newcommand{\dist}{\mathrm{dist}} 
\newcommand{\di}{\mathrm{div}} 
\newcommand{\diam}{\mathrm{diam}} 
\newcommand{\Lip}{\mathrm{Lip}} 

\newcommand{\abs}[1]{\lvert#1\rvert} 

\def\var{\mathbf{var}\,}

\def\E{\mathcal{E}}

\def\H{\mathcal{H}}

\def\HH{\mathbf{H}}

\def\SS{\mathbb S}
\def\N{\mathbb N}
\def\Z{\mathbb Z}

\def\R{\mathbb R}

\def\Om{\Omega}

\def\a{\alpha}
\def\b{\beta}

\def\de{\delta}
\def\e{\varepsilon}
\def\k{\kappa}
\def\l{\lambda}
\def\s{\sigma}
\def\om{\omega}
\def\vphi{\varphi}
\def\cl{{\rm cl}\,}

\def\Lip{{\rm Lip}}

\def\Div{{\rm div}\,}

\def\Id{{\rm Id}\,}
\def\dist{{\rm dist}}

\def\diam{{\rm diam}}
\def\spt{{\rm spt}}

\def\weakstar{\stackrel{*}{\rightharpoonup}}

\def\ov{\overline}
\def\pa{\partial}

\def\bd{{\rm bd}\,}

\title[Soap films with gravity and almost-minimal surfaces]{Soap films with gravity \\ and almost-minimal surfaces}

\author[Maggi]{F. Maggi}
\address[Francesco Maggi]{
\newline \indent Department of Mathematics, University of Texas at Austin,
\newline \indent 2515 Speedway STOP C1200, 78712, Austin, Texas, USA}
\email{maggi@math.utexas.edu}

\author[Scardicchio]{A. Scardicchio}
\address[Antonello Scardicchio]{
\newline \indent International Centre for Theoretical Physics,
\newline \indent Strada Costiera 11, 34151, Trieste, Italy}
\email{ascardic@ictp.it}

\author[Stuvard]{S. Stuvard}
\address[Salvatore Stuvard]{
\newline \indent Department of Mathematics, University of Texas at Austin,
\newline \indent 2515 Speedway STOP C1200, 78712, Austin, Texas, USA}
\email{stuvard@math.utexas.edu}

\begin{document}

\begin{abstract}
Motivated by the study of the equilibrium equations for a soap film hanging from a wire frame, we prove a compactness theorem for surfaces with asymptotically vanishing mean curvature and fixed or converging boundaries. In particular, we obtain sufficient geometric conditions for the minimal surfaces spanned by a given boundary to represent all the possible limits of sequences of almost-minimal surfaces. Finally, we provide some sharp quantitative estimates on the distance of an almost-minimal surface from its limit minimal surface.
%
\end{abstract}

\maketitle

\begin{center}{\it Dedicated to Luis Caffarelli, on his 70th birthday}\end{center}

\smallskip

\tableofcontents

\section{Introduction} In the study of soap films under the action of gravity, one is interested in surfaces with {\it small} but non-zero mean curvature spanned by a given boundary. Indeed, as explained in section \ref{section soap films with gravity} below, the mid-surface $M$ of a film of thickness $2\,h>0$ satisfies in first approximation the equilibrium condition
\begin{equation}
  \label{minimal surface with gravity h}
  H_M(x)=\k^2\,h\,\nu_M(x)\cdot e_3+O(h^2)\qquad\forall x\in M\,,
\end{equation}
where $H_M$ is the mean curvature of $M$ with respect to the unit normal $\nu_M$, $e_3$ is the vertical direction, and $\k^{-1}$ is the {\it capillary length} of the film, defined by
\begin{equation}
  \label{capillarity length}
  \k:=\sqrt{\frac{g\,\rho}\s}\,.
\end{equation}
Here, $\rho$ is the volume density of mass for the film solution, $\s$ denotes the surface tension of the film (with dimensions Newton per unit length), and $g$ is the gravity acceleration on Earth. The interest for this equation lies in the fact that it correctly encodes several physical properties which are missed by the minimal surface equation $H_M=0$, e.g. the fact that actual soap films cannot be formed under arbitrary large scalings of the boundary curve.

In this setting, the first question one wants to answer is whether minimal surfaces are a good model for their small mean curvature counterpart. In this paper, we provide a general sufficient condition on the boundary data to ensure the validity of this approximation. When the model minimal surface is smooth and strictly stable, we also provide quantitative estimates for almost-minimal surfaces in terms of their total mean curvature. Since formal statements require the introduction of a few concepts from Geometric Measure Theory, we present for the moment just an informal and simplified version of our main results.

\begin{theorem*} Let $\Gamma$ be a compact, orientable $(n-1)$-dimensional surface without boundary in $\R^{n+1}$, and let $\{M_j\}_j$ be a sequence of compact, orientable $n$-dimensional surfaces in $\R^{n+1}$ with boundaries $\Gamma_j=f_j(\Gamma)$ for maps $f_j$ converging in $C^1$ to the identity map, and such that (denoting by $\H^n$ the $n$-dimensional Hausdorff measure in $\R^{n+1}$),
\[
\sup_{j\in\N}\Big\{\max_{x\in M_j}|x|, \H^n(M_j)\Big\}<\infty\,,\qquad \lim_{j\to\infty}\int_{M_j}|H_{M_j}|\,d\H^n=0\,.
\]
Assume that $\Gamma$ has the following two properties:

\medskip

\noindent {\bf Finiteness and regularity of the Plateau problem:} There are finitely many minimal surfaces $\{N_i\}_i$ spanned by $\Gamma$, possibly including in the count ``singular'' minimal surfaces, whose singularities are anyway located away from $\Gamma$.

\medskip

\noindent {\bf Accessibility from infinity:} For each connected component $\Gamma'$ of $\Gamma$, the set of points $x\in\Gamma'$ such that, for some unit vectors $\nu_1$ and $\nu_2$ with $\nu_1\cdot\nu_2<1$, the inclusion
   \begin{equation}
     \label{acc}
   \Gamma\subset x+\Big\{y\in\R^{n+1}\,:\,y\cdot\nu_1\geq 0\,,y\cdot\nu_2\geq  0\Big\}
   \end{equation}
   holds, is a set of positive $\H^{n-1}$-measure; see
   \begin{figure}
     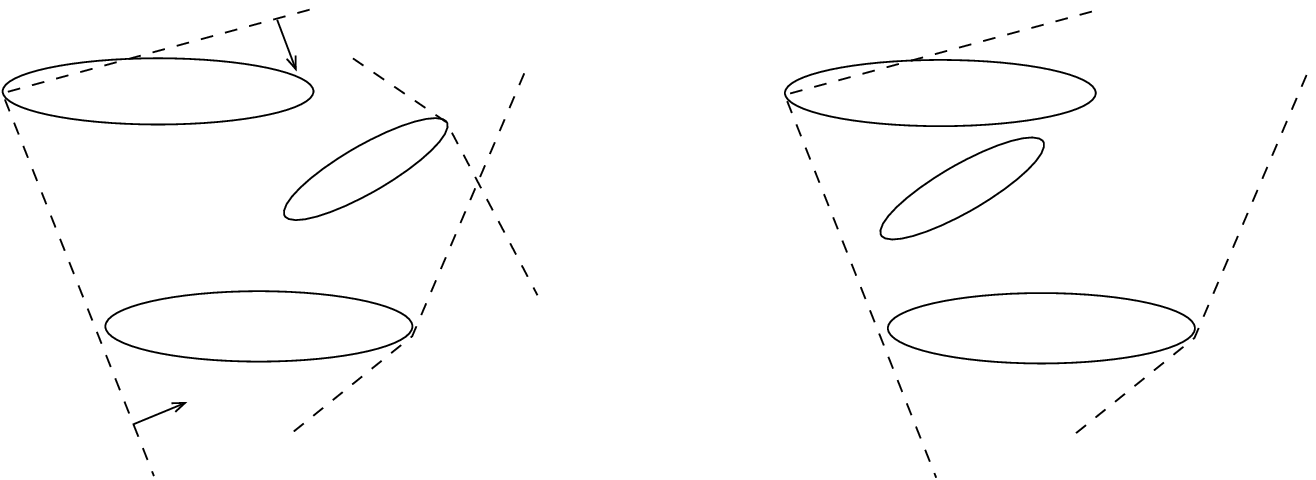\caption{{\small On the left, a boundary $\Gamma$, consisting of three circles, that is accessible from infinity. The acute wedges realizing the inclusions \eqref{acc} are depicted by dashed lines. Notice that it is not necessary that $\Gamma$ is contained into a convex set, or into a mean convex set, for the condition to hold. On the right, another set of circles defining a boundary $\Gamma$ which does not satisfy accessibility from infinity. Indeed, there is no way to touch the smaller circle with an acute wedge containing the larger ones.}}\label{fig accessible}
   \end{figure}
   Figure \ref{fig accessible}.

\medskip

\noindent Under these two assumptions, we have the following conclusions:

\medskip

\noindent {\bf No-bubbling:} There exists a single minimal surface $N_i$ such that $M_j\to N_i$ as $j\to\infty$, in the sense that there exist open sets $\{E_j\}_j$ with smooth boundary such that
\[
\lim_{j\to\infty}|E_j|+\H^n\Big(\pa E_j\setminus(N_i\cup M_j)\Big)=0\,.
\]
Here $|E|$ denotes the $(n+1)$-dimensional volume of $E\subset\R^{n+1}$.

\medskip

\noindent {\bf Strong convergence and sharp estimates:} If in addition $\Gamma_j=\Gamma$, $N_i$ has no singularities, and $N_i$ is strictly stable, in the sense that, for a positive constant $\l$,
\[
\int_{N_i} \abs{\nabla \varphi}^2 - \abs{A_{N_i}}^2 \varphi^2 \geq \lambda \int_{N_i} \varphi^2 \qquad \forall \varphi \in H^{1}_0(N_i)\,,
\]
(where $|A_{N_i}|$ is the Hilbert-Schmidt norm of the second fundamental form of $N_i \hookrightarrow \R^{n+1} $), and if for some $p>n$ we have a uniform bound
\[
\sup_{j\in\N}\int_{M_j}|H_{M_j}|^p\,d\H^n<\infty\,,
\]
then there exist smooth functions $u_j:N_i\to\R$ with $u_j=0$ on $\pa N_i$ and $\|u_j\|_{C^1(N_i)}\to 0$ as $j\to\infty$ such that
\[
M_j=\Big\{x+u_j(x)\,\nu_{N_i}(x):x\in N_i\Big\}\,,
\]
and the following sharp estimates hold:
\begin{eqnarray}\label{sharp 1}
  \|u_j\|_{C^0(N_i)}\le C\,\Big(\int_{M_j}|H_{M_j}|^p\Big)^{1/p}\,,
  \\\label{sharp 2}
  \max\Big\{\H^n(M_j)-\H^n(N_i),\|u_j\|_{W^{1,2}(N_i)}\Big\}\le C\,\Big(\int_{M_j}|H_{M_j}|^2\Big)^{1/2}\,,
\end{eqnarray}
for a constant $C=C(N,p)$.
\end{theorem*}

\begin{remark}
  {\rm As shown by simple examples (see Figure \ref{fig bubbling}), if accessibility from infinity fails then bubbling can occur in the convergence of $\{M_j\}_j$. In particular, $\{M_j\}_j$ could converge to a smooth minimal surface with multiplicity $2$, and some pieces of the limiting minimal surface could not be part of any minimal surface spanned by the whole $\Gamma$.}
\end{remark}

\begin{remark}
  {\rm In the case $M_j$ is the boundary of an open set (and thus, necessarily, $\Gamma=\emptyset$), and $M_j$ has {\it almost-constant} (non-zero) mean curvature, then the occurrence of bubbling is unavoidable, and its description has been undertaken in various papers, see e.g. \cite{breziscoron84,struwe84,ciraolomaggi2017,delgadinomaggimihailaneumayer,krummelmaggi,delgamaggi}. From this point of view, the fact that we can avoid bubbling under somehow generic assumptions on the boundary data $\Gamma$ is a remarkable rigidity feature of Plateau's problem.}
\end{remark}

\begin{remark}
  {\rm The finiteness and regularity assumption is well-illustated in the case when $\Gamma$ consists of two parallel unit circles in $\R^3$, having centers on a common axis. The idea here is that, depending on the distance between the circles, there should be at most five ``generalized'' minimal surfaces spanned by $\Gamma$ (see Figure \ref{fig catenoids}): two parallel disks, two catenoids (one stable, the other unstable), and two singular catenoids. Each singular catenoid is formed by attaching a smaller disk to two catenoidal necks so that the disk floats at mid distance from the two boundary circles, and the necks form three $120$-degrees angles along the circle. Notice that the floating circle does not count as a boundary curve, but rather as a curve of ``singular'' points. Observe that accessibility from infinity trivially holds in this case, while the validity of the finiteness and regularity assumption (which is formally introduced in section \ref{section finite plateau}) is not obvious, although it seems quite reasonable to expect it to be true. If that is the case, the compactness theorem indicates that a sequence of smooth almost-minimal surfaces spanned by $\Gamma$ (or with boundaries converging to $\Gamma$) must converge to one of these five minimal surfaces, without bubbling. Actually, a simple additional argument can be used to exclude that the singular catenoids are possible limits, see Remark \ref{rmk:classical_singularities}.}
\end{remark}

\begin{remark}
  {\rm Both estimates \eqref{sharp 1} and \eqref{sharp 2} are sharp. When $p=\infty$, \eqref{sharp 1} generalizes to arbitrary minimal surfaces the fact that an almost-minimal surface bounded by a circle deviates from a flat disk at most linearly in the mean curvature times the area of the disk. The interest of \eqref{sharp 2} is that the $L^2$-norm of the mean curvature appears as the dissipation of the area along a mean curvature flow with prescribed boundary data, see for example Huisken \cite{huiskenplateau} and Spruck \cite{spruck}. Moreover, we notice the close relation between \eqref{sharp 2} and the main result from \cite{dephilippismaggi}, which addresses the problem of proving global stability inequalities for smooth, area-minimizing surfaces. Finally, we remark that the bound on $\|H_{M_j}\|_{L^p(M_j)}$ for $p>n$ is needed to enforce the graphicality of $M_j$ over $N_i$ via Allard's regularity theorem. If one knows a priori that $M_j$ is a graph over $N_i$, then \eqref{sharp 1} can be proved for every $p\ge 2$ with $p>n/2$ (for example, $p=2$ works for two and three dimensional surfaces); see Theorem \ref{thm graphs} in section \ref{section decay rates} below.}
\end{remark}

The paper is organized as follows. In section \ref{section soap films with gravity} we discuss the equilibrium conditions for soap films with gravity, and derive \eqref{minimal surface with gravity h} under appropriate conditions. An interesting outcome of this discussion is the idea, based on physical grounds,  of formulating Plateau's problem {\bf as a singular capillarity problem}. Section \ref{section almost minimal surfaces} consists in part of a preliminary review of the necessary concepts from Geometric Measure Theory, and in part of a precise formulation of our two main assumptions. In section \ref{section proof of compactness thm} we give a precise statement and the proof of our main compactness result, see Theorem \ref{thm main}. Finally, in section \ref{section decay rates}, we explain the reduction to graph-like surfaces, and prove various sharp convergence estimates, see Theorem \ref{thm graphs}. These last results show that on graph-like surfaces one can work with a very weak notion of almost-minimality deficit, a fact that will likely prove useful in future investigations.

\bigskip

\noindent{\bf Acknowledgments.} F.M. and S.S. have been supported by the NSF Grants DMS-1565354, DMS-1361122 and DMS-1262411.

\section{Soap films with gravity}\label{section soap films with gravity} Due to gravitational forces, surfaces with small but non-zero mean curvature arise naturally in the study of soap films hanging on a wire. This effect is usually neglected in the mathematical literature, leading to an exclusive focus on minimal surfaces. The resulting model describes correctly the physical situation of small soap films. However, as noticed by Defay and Prigogine, ``gravitational forces [...] play a dominant role in determining the shapes of \emph{macroscopic} surfaces''; see \cite[Section I-4]{DefayPrigogineBOOK}. The typical length scale which separates small films from large films is given by the capillary length $\k^{-1}=\sqrt{\sigma/\rho g}$, introduced in \eqref{capillarity length}. For a solution of soap in water at room temperature, the values of the surface tension and of the density are, respectively, $\sigma\simeq 0.03$N/m and $\rho  \simeq 10^3$kg/m$^3$, while $g  \simeq 9.81$N/kg is Earth's gravity, so that the length-scale $\k^{-1}$ is of order of $1.7$mm. The deviation of a soap film with gravity from its limit minimal surface is expected to be $O(h\,\k)$ where $h$ is the average width of the film. For typical soap films, we are in the perturbative region, since we usually have $h\simeq 10^{-3} \mathrm{mm}\simeq 10^{-3}\k^{-1}$.

Idealizing the wire frame as a smooth curve $\Gamma$ without boundary in $\R^3$, and the soap film as a smooth surface $M$ bounded by $\Gamma$, if we neglect gravity then we are led to modeling soap films as minimal surfaces, i.e. surfaces with vanishing mean curvature
\begin{equation}
  \label{minimal surface}
  H_M=0\,.
\end{equation}
This condition is derived from balancing the atmospheric pressures on the two sides of the film with the {\it Laplace pressure} induced by surface tension \cite{Young1805,Laplace1806}. Denoting by $\s$ the surface tension, if $S$ is a small neighborhood of $x\in M$, with outer unit co-normal $\nu_S^M$ with respect to $M$, then the tension on $S$ is given by
\begin{equation}
  \label{from tension to pressure}
  \s\,\int_{\pa\,S}\nu_S^M=\s\,\int_S\,\HH_M\,.
\end{equation}
Here, $\HH_M$ denotes the mean curvature vector to $M$, which, once the choice of a unit normal $\nu_M$ to $M$ is specified, defines a scalar mean curvature $H_M$ appearing in \eqref{minimal surface} through the equation $\HH_M=H_M\,\nu_M$. If the atmospheric pressures on the two sides of the film are assumed to be equal, as it is the case if we ignore gravity, then the Laplace pressure must vanish, and we find \eqref{minimal surface}. Let us recall that \eqref{minimal surface} can also be derived by the principle of virtual works, as first done by Gauss \cite{Gauss1830}, by taking as the total energy of the film the area of $M$ times $\s$, namely
\begin{equation}
  \label{gauss no gravity}
  \E(M)=\s\,\H^2(M)\,.
\end{equation}
Equation \eqref{minimal surface} fails in describing macroscopic soap films in two ways:

\medskip

\noindent (i) For a given contour $\Gamma$, the minimal surfaces spanned by $t\,\Gamma$, for a scaling factor $t>1$, are simply obtained by scaling the minimal surfaces spanned by $\Gamma$. This is evidently not the case for real soap films, where there is a competition between the capillary length $\kappa^{-1}$ and the length-scale of the boundary curve $\Gamma$ in determining if a soap film is produced at all. From this point of view, $H_M=0$ fails completely at describing the macroscopic length-scales at which soap films are actually formed. Equation \eqref{minimal surface with gravity h}, namely $H_M=\k^2\,h\,\nu_M\cdot e_3+O(h^2)$, does not have this problem. Indeed, the solvability of a prescribed mean curvature equation $H_M=f$ with $\pa M =\Gamma$ requires a control on the size of $f$ in terms, for example, of $\H^2(M_{\Gamma})^{-1/2}$, where $M_\Gamma$ is the area-minimizing surface spanned by $\Gamma$; see, e.g., the papers by Duzaar and Fuchs \cite{duzaarfuchsMANU,duzaarfuschBUMI}. In particular, the solvability of \eqref{minimal surface with gravity h} with boundary condition $\pa M=\Gamma$ depends on the relative sizes of $\k^2\,h$ (which measures the physical properties of the soap solution) and of the length-scale of $\Gamma$.

\medskip

\noindent (ii) Equation \eqref{minimal surface} is invariant under rotations, while the effect of gravity is definitely anisotropic. For example, a soap film $M$ hanging from a circular frame $\Gamma$ of radius $r$ should be exactly a flat disk if $\Gamma$ is contained in a vertical plane, whereas it should possess a non-trivial curvature if $\Gamma$ is in horizontal position, with average vertical deviation from the flat disk of order $r^2\,H_M$. This deviation is observable depending on the length scale of $\Gamma$ and on $\k$. In the case of soap {\it bubbles}, where $H_M=0$ is replaced by $H_M$ constant, a deviation is experimentally observed and is substantial; see \cite[Figures 1 and 3]{Cohen2515}. The presence of the vertical component of $\nu_M$ makes indeed \eqref{minimal surface with gravity h} anisotropic, and, actually,  \eqref{minimal surface with gravity h} boils down to \eqref{minimal surface} only if $M$ is contained in a vertical plane.

\medskip

In order to take the effect of gravity into account, one might be tempted to add to the surface tension energy functional a term corresponding to the potential energy of the film, namely, to consider
\begin{equation}
  \label{gauss with gravity wrong}
\E(M) =  \s\,\H^2(M)+g\,\rho_*\int_M\,x_3\,dx
\end{equation}
in place of \eqref{gauss no gravity}, with $\rho_*$ denoting surface density of mass. While this would be correct for a solid elastic slab, or a rubber sheet, for a fluid it is clearly incorrect. In fact, it would amount to replace $H_M=0$ with the  equation $H_M(x)=\k^2\,x_3$, which would incorrectly predict that a soap film hanging from a perfectly planar wire contained in a vertical plane should have curvature and lie out of the plane!

In \cite[Section I.4]{DefayPrigogineBOOK}, Defay and Prigogine explain how the effect of gravity should be modeled by balancing pressures. One needs to consider the finite thickness of the film, bounded by two different interfaces, and to take into account the difference in hydrostatic pressures on the two faces caused by the gravitational pull. We now put into equations this idea, and formulate a PDE for the problem. The resulting PDE, see \eqref{minimal surface with gravity}, justifies \eqref{minimal surface with gravity h}, which, in turn, appears in the literature when $M$ is axially symmetric and very close (in a $C^1$-sense) to a plane; see e.g. \cite[Equation (2.5)]{deGennesBOOK}.

Consider a smooth two-dimensional surface $M$ bounded by a smooth curve $\Gamma$ in $\R^3$, and oriented by a unit normal $\nu_M$. Here $M$ plays the role of an ideal surface lying inside the film. Given a smooth function $\a$ defined on $M$, we denote its graph over $M$ by
\[
M(\a):=\Big\{x+\a(x)\,\nu_M(x):x\in M\Big\}\,.
\]
The two interfaces of the soap film are described by graphs $M(\a)$ and $M(-\beta)$ for positive functions $\a$ and $\beta$. Up to replacing $M$ with $M((\a-\b)/2)$, and then setting $\psi:=(\a+\b)/2$, we can actually assume that the interfaces are $M(\psi)$ and $M(-\psi)$, where $\psi$ is a smooth positive function on $M$. However, it does not seem that the symmetric parametrization is always the most convenient, so we shall argue in terms of $\a$ and $\beta$.

Given $x\in M$, and with reference to Figure \ref{fig laplace},
\begin{figure}
  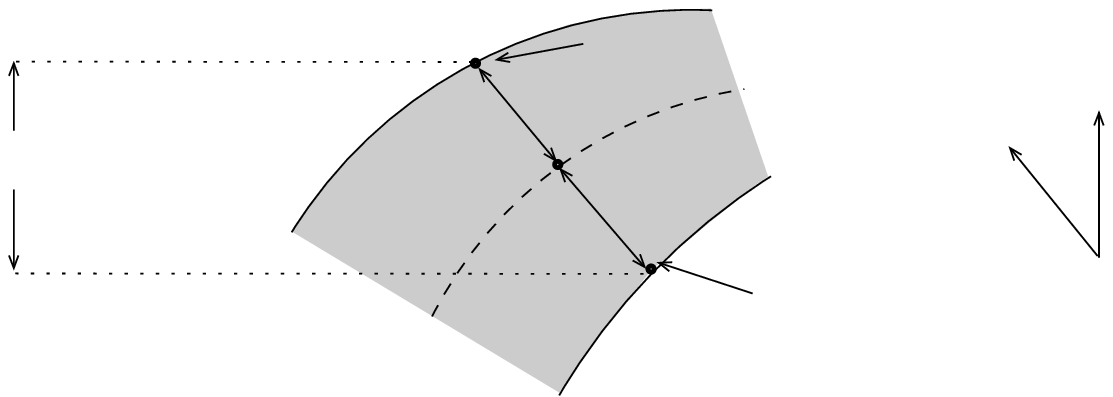\caption{{\small The derivation of \eqref{minimal surface with gravity}, after \cite[Section I.4]{DefayPrigogineBOOK}.}}\label{fig laplace}
\end{figure}
at equilibrium, the pressure $p(x^+)$ at $x^+:=x+\a(x)\,\nu_M(x)\in M(\a)$ is given by
\begin{equation}\label{laplace at x+}
p(x^+)=p_0-\s\,H_{M(\a)}(x^+)\,,
\end{equation}
where $H_{M(\a)}$ is the scalar mean curvature of $M(\a)$ with respect to the unit normal pointing {\it outside} the film, $p_0$ is the atmospheric pressure, and $\s$ is the surface tension. The pressure $p(x^-)$ at $x^-:=x-\beta(x)\,\nu_M(x)\in M(-\beta)$ is similarly given by
\begin{equation}
  \label{laplace at x-}
  p(x^-)=p_0-\s\,H_{M(-\beta)}(x^-)\,,
\end{equation}
where $H_{M(-\beta)}$ is the scalar mean curvature of $M(-\beta)$ with respect to the unit normal pointing {\it outside} of the film. Subtracting the two equations, we obtain
\[
H_{M(\a)}(x^+)-H_{M(-\beta)}(x^-)=\frac{p(x^-)-p(x^+)}\s\,.
\]
The difference between $p(x^-)$ and $p(x^+)$ is the hydrostatic pressure
\begin{equation}
  \label{hydrostatic pressure}
  \begin{split}
  &p(x^-)-p(x^+)=g\,\rho\,(x^+-x^-)\cdot e_3=g\,\rho\,(\a(x)+\beta(x))\,\nu_M^{(3)}(x)\,,
  \\
  &\mbox{where}\quad \nu_M^{(3)}:=\nu_M\cdot e_3\,.
  \end{split}
\end{equation}
Combining \eqref{laplace at x+}, \eqref{laplace at x-} and \eqref{hydrostatic pressure} we obtain the equation for minimal surfaces with gravity
\begin{equation}
  \label{minimal surface with gravity}
  H_{M(\a)}(x^+)-H_{M(-\beta)}(x^-)=\k^2\,(\a(x)+\beta(x))\,\nu_M^{(3)}(x)\,,\qquad\forall x\in M\,.
\end{equation}
If $|\nabla\a|$ and $|\nabla\b|$ are sufficiently small at $x$, and we consider the mid-surface parametrization, then we can assume that locally $\a\equiv\beta\equiv h$, where $h$ is a small positive constant. Denoting by $\{\k_1,\k_2\}$ the principal curvatures of $M$, and stressing the smallness of $h$ by requiring $0<h<\max\{|\k_1|,|\k_2|\}^{-1}$, we thus obtain
\begin{eqnarray*}
H_{M(\a)}(x^+)&=&\sum_{i=1}^2\frac{\k_i(x)}{1+h\,\k_i(x)}=H_M(x)-h\,\sum_{i=1}^2\k_i^2+O(h^2)\,\sum_{i=1}^2\k_i^3\,,
\\
H_{M(-\beta)}(x^-)&=&-\sum_{i=1}^2\frac{\k_i(x)}{1-h\,\k_i(x)}=-H_M(x)-h\,\sum_{i=1}^2\k_i^2+O(h^2)\,\sum_{i=1}^2\k_i^3\,,
\end{eqnarray*}
and \eqref{minimal surface with gravity} is readily seen to imply
\[
H_M(x)=\k^2\,h\,\nu_M^{(3)}(x)+O(h^2)\qquad\forall x\in M\,,
\]
that is, \eqref{minimal surface with gravity h}.

\bigskip

We now explain how \eqref{minimal surface with gravity} can be derived from energy considerations. The idea is treating the problem of a soap film hanging from a wire frame as a capillarity problem. We model the wire frame as a solid $\de$-neighborhood of an idealized curve $\Gamma$, setting
\[
\Gamma_\de:=\Big\{x\in\R^3:\dist(x,\Gamma)\le\de\Big\}\,,\qquad A_\de:=\R^3\setminus\Gamma_\de\,.
\]
We model the soap film as a set $E\subset A_\de$ with very small volume $\e=|E|$, and, following Gauss' treatment of capillarity theory, we define its energy as
\[
\E(E)=\s\,\H^2\big(A_\de\cap\pa E\big)+\s\,\gamma\,\H^2(\pa A_\de\cap\pa E)+g\,\rho\,\int_E\,x_3\,dx\,,
\]
see
\begin{figure}
  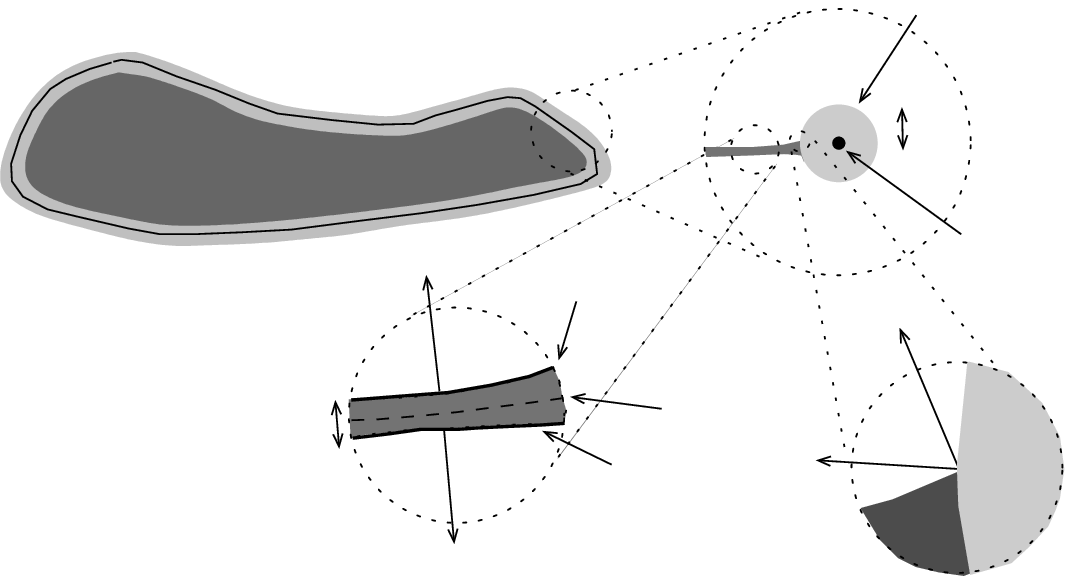\caption{{\small Using Gauss' capillarity energy to formulate Plateau's problem. Minimization of $\s\,\H^2(M)$ among surfaces with $\pa M=\Gamma$ is replaced by minimizing the capillarity energy among regions contained in the complement of a $\de$-neighborhood of $\Gamma$. Equilibrium configurations with volume $\e\ll\de\,\H^2(S)\ll1$ arise as normal neighborhoods of minimal surfaces spanned by $\Gamma$. Here $S$ denotes the boundary of $E$ away from the wire frame.}}\label{fig capillarity}
\end{figure}
Figure \ref{fig capillarity}. Here $\gamma\in(-1,1)$ is a dimensionless parameter taking into account the ratio between the surface tension on the liquid-air interface, and the surface tension on the liquid-solid interface along the wire frame walls. Assuming that $E$ is a smooth critical point of this energy, the Euler-Lagrange equations boil down to the equilibrium condition
\begin{equation}
  \label{gauss equilibrium}
  \s\,H_E(y)+\rho\,g\,y_3=\l\qquad\forall y\in S=A_\de\cap\pa E\,,
\end{equation}
where $H_E$ denotes the scalar mean curvature of $\pa E$ with respect to the {\it outer} unit normal to $E$, and $\l$ is a Lagrange multiplier associated to the volume constraint. Equation \eqref{gauss equilibrium} is coupled with {\it Young's law},
\begin{equation}
  \label{youngs law}
  \nu_E(y)\cdot\nu_{\Gamma_\de}(y)=\gamma\qquad \forall y\in \ov{S}\cap\pa A_\de\,.
\end{equation}
Under the assumption that
\[
\frac{\e}{\H^2(S)}\ll\de
\]
and that $\de$ is sufficiently small in terms of the local and global geometric properties of $\Gamma$, it is reasonable to expect the existence of critical points $E$ described by means of mid-surfaces $M$ spanned by $\Gamma$. More precisely, we consider critical points $E$ corresponding to surfaces $M$ with $\pa M=\Gamma$ in the sense that, for every $x\in M\cap A_\de$ we can find $r>0$ such that
\[
E\cap B_r(x)=\Big\{z+t\,\nu_M(z):z\in M\,,-\beta(z)<t<\a(z)\Big\}\,.
\]
In this case, \eqref{gauss equilibrium} computed at $y=x+\a(x)\,\nu_M(x)=x^+$ and at $y=x-\b(x)\,\nu_M(x)=x^-$ gives
\begin{equation}
  \label{el eqn}
  \s\,H_{M(\a)}(x^+)+\rho\,g\,x^+\cdot e_3=\l\,,\qquad \s\,H_{M(-\b)}(x^-)+\rho\,g\,x^-\cdot e_3=\l\,.
\end{equation}
Notice that our sign conventions on scalar mean curvatures have been such that $H_E(x^+)=H_{M(\a)}(x^+)$ and $H_E(x^-)=H_{M(-\b)}(x^-)$. Subtracting the two equations we deduce indeed the validity of \eqref{minimal surface with gravity} as a consequence of the equilibrium condition for Gauss' capillarity energy. Notice that the full set of equilibrium conditions is expressed by considering Young's law together with the two equations \eqref{el eqn}, or with the single equation \eqref{gauss equilibrium}, rather than by \eqref{minimal surface with gravity} alone. Here the role of \eqref{minimal surface with gravity} is stressed because, as explained above, it clearly motivates the study of surfaces with small mean curvature.

\bigskip

In summary, we have seen in this section how surfaces with prescribed boundary and small mean curvature, such as the ones described by equation \eqref{minimal surface with gravity}, or by its approximation \eqref{minimal surface with gravity h}, arise naturally in the study of soap films hanging from a wire. More generally, the use of capillarity theory to model soap films provides an additional, more physical, point of view on the long-debated issue of prescribing boundary data in the mathematical formulation of Plateau's problem; see \cite{jenny,davidguy,harrisonpughADVCALCVAR,dPdRghira,delellisghiraldinmaggiJEMS,giusteri,delederosaghira,bellettini,fangkola,derosaSIAM} for the most recent developments on this venerable question. Leaving a more complete discussion of this last point to a forthcoming paper, we focus here on a first problem raised by this approach, namely understanding the relation between almost-minimal and minimal surfaces.


\section{Almost-minimal surfaces}\label{section almost minimal surfaces} Let $\Gamma$ be a compact $(n-1)$-dimensional surface in $\R^{n+1}$ without boundary. Motivated by the study of surfaces obeying \eqref{minimal surface with gravity h}, we now consider the general question of understanding the relation between the minimal and the almost minimal surfaces spanned by $\Gamma$.
The question we want to address is the following:
\begin{equation}
  \label{question}
  \begin{split}
  &\mbox{In the class of surfaces spanned by $\Gamma$}
  \\
  &\mbox{is the family of minimal surfaces rich enough}
  \\
  &\mbox{to describe all the possible limits of almost-minimal surfaces?}
\end{split}
\end{equation}
Theorem \ref{thm main} answers affirmatively to this question under the assumptions that $\Gamma$ is accessible from infinity and spans finitely many minimal surfaces without boundary singularities. The statement of the theorem is actually quite delicate, as it involves several choices and assumptions. In the following paragraphs we shall address these points. In {\S\, \ref{section deficits}} we propose various ways of measuring the almost-minimality of a surface, while in {\S\, \ref{section convergence}} we review two notions of convergence for smooth surfaces arising in Geometric Measure Theory. In {\S\, \ref{section infinity}} we discuss our geometric assumption on the connected components of $\Gamma$, and in {\S\, \ref{section finite plateau}} we make precise the idea that $\Gamma$ spans at most finitely many minimal surfaces.

\subsection{Measuring almost-minimality}\label{section deficits} Directly motivated by the equation for minimal surfaces with gravity \eqref{minimal surface with gravity h}, we shall consider the {\bf uniform deficit}
\[
\de_\infty(M):=\|H_M\|_{C^0(M)}
\]
as our chief option to measure almost-minimality. But depending on other possible applications of almost-minimal surfaces, the family of {\bf integral  deficits}
\[
\de_p(M):=\|H_M\|_{L^p(M)}\,,\qquad 1\le p<\infty\,,
\]
may be more relevant. For example, $\de_2(M)$ definitely plays a role in the study of the gradient flow defined by Plateau's problem, see \cite{huiskenplateau,spruck}. At the weaker end of the spectrum, and closer to the point of view usually adopted when discussing Paley-Smale sequences in variational problems, one may consider the {\bf duality deficits}
\[
\de_{-p}(M):=\sup\Big\{\int_M\,\Div^M X\,d\H^n:X\in C^1_c(\R^{n+1}\setminus\Gamma;\R^{n+1})\,,\|\nabla X\|_{L^p(\R^{n+1})}\le 1\Big\}\,,
\]
for $1\le p\le\infty$. This last definition is motivated by the tangential divergence theorem, stating that if $M$ is a smooth compact $n$-dimensional surface with boundary $\Gamma$, then
\begin{equation}
  \label{tangential divergence theorem}
  \int_M\,\Div^M X\,d\H^n=\int_M\,X\cdot\HH_M\,d\H^n+\int_\Gamma\,X\cdot\nu_\Gamma^M\,d\H^{n-1}\qquad\forall X\in C^1_c(\R^{n+1};\R^{n+1})\,.
\end{equation}
Here $\nu_\Gamma^M$ is outer unit co-normal to $\Gamma$ with respect to $M$, and $\Div^MX$ is the tangential divergence of $X$ with respect to $M$, that is
\begin{equation}
  \label{tangential divergence of X}
  \Div^MX(x):=\Div X(x)-\nu_M(x)\cdot\nabla X(x)[\nu_M(x)]\qquad\forall x\in M\,.
\end{equation}
An interesting fact is that on surfaces $M$ that are a priori known to be graphs over strictly stable minimal surfaces, the duality deficit $\de_{-\infty}(M)$ already controls the area deficit, see Theorem \ref{thm graphs}.

\subsection{Convergence of smooth surfaces} \label{section convergence} In order to provide a better insight into question \eqref{question}, we need to discuss possible notions of limit for a sequence of smooth surfaces. To introduce the relevant ideas, let us consider a sequence $\{M_j\}_j$ of smooth oriented $n$-dimensional surfaces such that
\begin{equation}
  \label{Mj assumption}
  \pa M_j=\Gamma\,,\qquad \sup_j\,\H^n(M_j)<\infty\,,\qquad\sup_j\sup_{x\in M_j}|x|<\infty\,,\qquad\lim_{j\to\infty}\de_\infty(M_j)=0\,.
\end{equation}
Geometric Measure Theory provides two canonical ways to discuss the convergence of such a sequence $\{M_j\}_j$. Both approaches require the identification of each $M_j$ as a linear functional on a space of test functions, or, equivalently, as a Radon measure on a suitable finite dimensional space. The first approach, the theory of currents, allows to transfer the spanning information $\pa M_j=\Gamma$ to a generalized limit surface. The second approach, the theory of varifolds, allows to infer from $\de_\infty(M_j)\to 0$ the existence of a limit surface that is minimal, again in a generalized sense. A subtlety lies in the fact that the generalized limit surface in the varifold sense may be larger that its counterpart in the sense of currents.

\medskip

\noindent {\it The viewpoint of currents}. We see each oriented surface $M_j$ in \eqref{Mj assumption} as a linear continuous functional $\llbracket M_j \rrbracket$ on the space $\D^n(\R^{n+1})$ of smooth, compactly supported $n$-dimensional differential forms, equipped with the standard topology of test functions. More precisely, if $M_{j}$ is oriented by a continuous choice of a unit normal vector field $\nu_{M_j}$, we set
\[
\langle\llbracket M_j \rrbracket,\omega\rangle:=\int_{M_j}\langle\star\nu_{M_j}(x),\omega(x)\rangle\,d\H^n(x)\,\qquad\forall\om\in\D^n(\R^{n+1})\,,
\]
where, given $\nu\in\SS^n$, $\star\nu$ denotes the simple unit $n$-vector corresponding to the $n$-dimensional plane $\nu^\perp$ oriented by $\nu$, and the duality between $n$-vectors and $n$-covectors appears under the integral. Let us recall that $\star\nu_{M_j}$ induces a smooth orientation $\tau_\Gamma$ on $\Gamma$ (that is, a smooth field of simple unit $(n-1)$-vectors defining and orienting the tangent planes to $\Gamma$) in such a way that Stokes' theorem holds
\begin{equation}
  \label{stokes theorem}
  \int_{M_j}\langle\star\nu_{M_j},d\omega\rangle\, d\H^n=\int_{\Gamma}\langle\tau_\Gamma,\omega\rangle\, d\H^{n-1}\qquad\forall\omega\in\D^{n-1}(\R^{n+1})\,,
\end{equation}
where $d\omega$ is the exterior differential of the $(n-1)$-form $\omega$. In this setting, it is quite natural to define the ``boundary'' of $\llbracket M_j \rrbracket$ as the linear continuous functional defined on $\D^{n-1}(\R^{n+1})$ by setting
\begin{equation}
  \label{boundary Mj}
\langle\pa\llbracket M_j \rrbracket,\omega\rangle :=\langle \llbracket M_j \rrbracket,d\omega\rangle\,\qquad\forall\om\in\D^{n-1}(\R^{n+1})\,.
\end{equation}
Of course, Stokes' theorem \eqref{stokes theorem} implies that if $\Gamma$ is oriented by the orientation $\tau_\Gamma$ induced by the choice of $\nu_{M_j}$ then
\[
\pa\llbracket M_j \rrbracket=\llbracket \Gamma \rrbracket\,.
\]
The second and the third condition in \eqref{Mj assumption} and the compactness theorem for Radon measures imply the existence of a linear continuous functional $T$ on $\D^n(\R^{n+1})$ such that, up to extracting subsequences,
\begin{equation}
  \label{who is T}
  \langle T,\omega\rangle=\lim_{j\to\infty}\int_{M_j}\langle\star\nu_{M_j},\omega\rangle\,d\H^n\,\qquad\forall\om\in\D^n(\R^{n+1})\,.
\end{equation}
Is the linear functional $T$ still represented by the action on forms of an oriented surface with boundary, like the functionals $\llbracket M_j \rrbracket$ are? A deep theorem of Federer and Fleming \cite{federerfleming60} gives a positive answer, provided that we introduce a suitable class of generalized surfaces with boundary. The key notion here is that of a {\it rectifiable set}. We say that a Borel set $N\subset\R^{n+1}$ is {\bf locally $\H^n$-rectifiable} if, up to a $\H^n$-null set, $N$ can be covered by countably many Lipschitz images of $\R^n$ into $\R^{n+1}$, and if $\H^n(N\cap B_R)<\infty$ for every $R>0$. If $N$ is locally $\H^n$-rectifiable, then $N$ has a tangent plane almost-everywhere, in the sense that for $\H^n$-a.e. $x\in N$ there exists an $n$-dimensional linear subspace $T_{x}N \subset \R^{n+1}$ such that
\begin{equation}
  \label{approximate tangent plane}
  \lim_{r\to 0^+}\int_{(N-x)/r}\vphi\,d\H^n=\int_{T_xN}\vphi\,d\H^n\qquad\forall \vphi\in C^0_c(\R^{n+1})\,.
\end{equation}
We can thus define a Borel vector field $\nu_{N}$ on $N$ such that $\nu_N(x)^\perp=T_xN$ at $\H^n$-a.e. $x \in N$. Analogously to the smooth setting, such a vector field $\nu_N$ will be called an {\it orientation} of the rectifiable set $N$. Coming back to \eqref{who is T}, the Federer--Fleming compactness theorem shows the existence of a locally $\H^n$-rectifiable set $N$, of a Borel measurable orientation $\nu_N$, and of a function $\alpha \in L^{1}_{loc}(\H^n \llcorner N; \Z)$ (an integer-valued {\it multiplicity} on $N$) such that $T=\llbracket N, \star \nu_N, \alpha \rrbracket$, i.e.
\begin{equation}
  \label{T is N alfa}
  \langle T,\om\rangle=\int_N\,\alpha(x)\,\langle\star\nu_N(x),\omega(x)\rangle\,d\H^n(x)\,\qquad\forall\omega\in \D^n(\R^{n+1})\,.
\end{equation}
Moreover, as a simple by-product of \eqref{boundary Mj}, we see that the limit current $T$ has still boundary $\Gamma$, in the sense that $\pa T=\llbracket\Gamma\rrbracket$, or, more explicitly:
\begin{equation}
  \label{partial T is}
  \langle T,d\om\rangle=\langle\llbracket\Gamma\rrbracket,\om\rangle\,\qquad\forall\omega\in \D^{n-1}(\R^{n+1})\,.
\end{equation}

\bigskip

\noindent {\it The viewpoint of varifolds}. The next question is if the rectifiable set $N$, found by taking the limit of $\{M_j\}_j$ in the sense of currents, is minimal, at least in some generalized sense. The starting point is the tangential divergence theorem applied on $M_j$ to fields supported away from $\Gamma$, which yields
\begin{equation}
  \label{lhs rhs}
\int_{M_j}\Div^{M_j}X\,d\H^n=\int_{M_j}X\cdot\HH_{M_j}\,d\H^n\qquad\forall X\in C^1_c(\R^{n+1}\setminus\Gamma;\R^{n+1})\,.
\end{equation}
Notice that, since $\de_\infty(M_j)\to 0$, the right-hand side of \eqref{lhs rhs} converges to zero as $j\to\infty$. To pass to the limit on the left-hand side we adopt the following point of view. Let us set
\[
G^n:=\R^{n+1}\times(\SS^n/\equiv)\,,
\]
where $\nu_1\equiv \nu_2$ if and only if $\nu_1=\pm\nu_2$, and denote by $[\nu]$ the $\equiv$-equivalence class of $\nu\in\SS^n$. The point $(x,[\nu])\in G^n$ identifies the (unoriented) $n$-dimensional affine plane orthogonal to $\nu$ and passing through $x$ in $\R^{n+1}$. Given $X\in C^1_c(\R^{n+1};\R^{n+1})$ we can define
\[
\vphi_X\in C^0_c(G^n)
\]
by setting
\[
\vphi_X(x,[\nu]):=\Div X(x)-\nu\cdot\nabla X(x)\nu\qquad (x,[\nu])\in G^n\,.
\]
The definition is well-posed, as the right-hand side is invariant when exchanging $\nu$ with $-\nu$. In this way
\[
\int_{M_j}\Div^{M_j}X\,d\H^n=\langle \var(M_j),\vphi_X\rangle
\]
if we agree to associate every smooth surface $M$ with a linear functional $\var(M)$ on $C^0_c(G^n)$ by setting
\begin{equation}
  \label{M as a varifold}
  \langle \var(M),\vphi\rangle :=\int_M\,\vphi(x,[\nu_M(x)])\,d\H^n(x)\qquad\forall\vphi\in C^0_c(G^n)\,.
\end{equation}
Notice that $M$ does not need to be orientable here, as we are considering $[\nu_M(x)]$ in \eqref{M as a varifold}. Clearly, $\var(M)$ can be seen as a Radon measure on $G^n$, with total mass equal to $\H^n(M)$. Thus, under the assumptions in \eqref{Mj assumption}, $\{\var(M_j)\}_j$ is a bounded sequence of Radon measures with uniformly bounded supports, so that the standard compactness theorem for Radon measures ensures the existence of a Radon measure $V$ on $G^n$ such that, up to extracting subsequences,
\begin{equation}
  \label{who is V}
\langle V,\vphi\rangle=\lim_{j\to\infty}\int_{M_j}\,\vphi(x,[\nu_{M_j}(x)])\,d\H^n(x)\qquad\forall\vphi\in C^0_c(G^n)\,.
\end{equation}
Given that $\de_\infty(M_j)\to 0$, the above argument shows that $\langle V,\vphi_X\rangle=0$ for every $X$ compactly supported in the complement of $\Gamma$. We then ask the question whether the varifold $V$ can be associated to a generalized surface, and to what extent this surface is minimal. Another deep theorem, this time due to Allard \cite{Allard}, provides the following answer: there exists a locally $\H^n$-rectifiable set $N$ and a function $\theta \in L^{1}_{loc}(\H^{n} \llcorner N;\N)$ (a non-negative integral {\it multiplicity} on $N$) such that $V$ is represented by $N$ and $\theta$, in symbols $V=\var(N,\theta)$, in the sense that
\begin{equation}
  \label{V is q N}
  \langle V,\vphi\rangle=\langle\var(N,\theta),\vphi\rangle=\int_{N}\,\theta(x)\,\vphi(x,[\nu_{N}(x)])\,d\H^n(x)\,\qquad\forall\vphi\in C^0_c(G^n)\,.
\end{equation}
As noticed, under the assumption \eqref{Mj assumption}, we have $\langle V,\vphi_X\rangle=0$ whenever $\spt X\cap\Gamma=\emptyset$. In other words, the varifold $V=\var(N,\theta)$ is {\it minimal} on $\R^{n+1}\setminus\Gamma$ (or {\it stationary}, in the common terminology of Geometric Measure Theory), in the sense that
\begin{equation}
  \label{minimal varifold}
  \int_{N} \theta\,\Div^{N} X\,d\H^n=0\qquad\forall X\in C^1_c(\R^{n+1}\setminus\Gamma;\R^{n+1})\,.
\end{equation}
Two remarks are in order: (i) The rectifiable set $N$ arising in the varifold convergence is in general {\it larger} than the rectifiable set $N$ obtained by taking the limit of $\{M_j\}_j$ in the sense of currents. The typical example is obtained by considering $M_j=B_1\cap(K/j)$ (for $j\to\infty$) where $K$ is a fixed catenoid. In this case the limit in the sense of currents is trivial, $N=\emptyset$, because the two sheets of the catenoid cancel out in the limit due to their opposite orientations; at the same time, if the limit is taken in the sense of varifolds, $N$ is equal to a unit disk with multiplicity $\theta=2$. For an example with fixed boundary data, see Example \ref{current vs varifold} below. From this point of view, answering question \eqref{question} partly amounts to \emph{determine conditions under which this ambiguity between the two limits, one taken in the sense of currents and the other in the sense of varifolds, does not occur};  (ii) Coming back to the generalized minimal surface condition \eqref{minimal varifold}, in the next classical example we notice how this condition allows one to include in the theory of minimal surfaces non-smooth examples that are actually physically relevant.

\begin{example}\label{example two rings part 1}
  {\rm Let $\Gamma=\Gamma_1\cup\Gamma_2$ be given by two parallel circles in $\R^3$ with centers on a same axis.
  \begin{figure}
    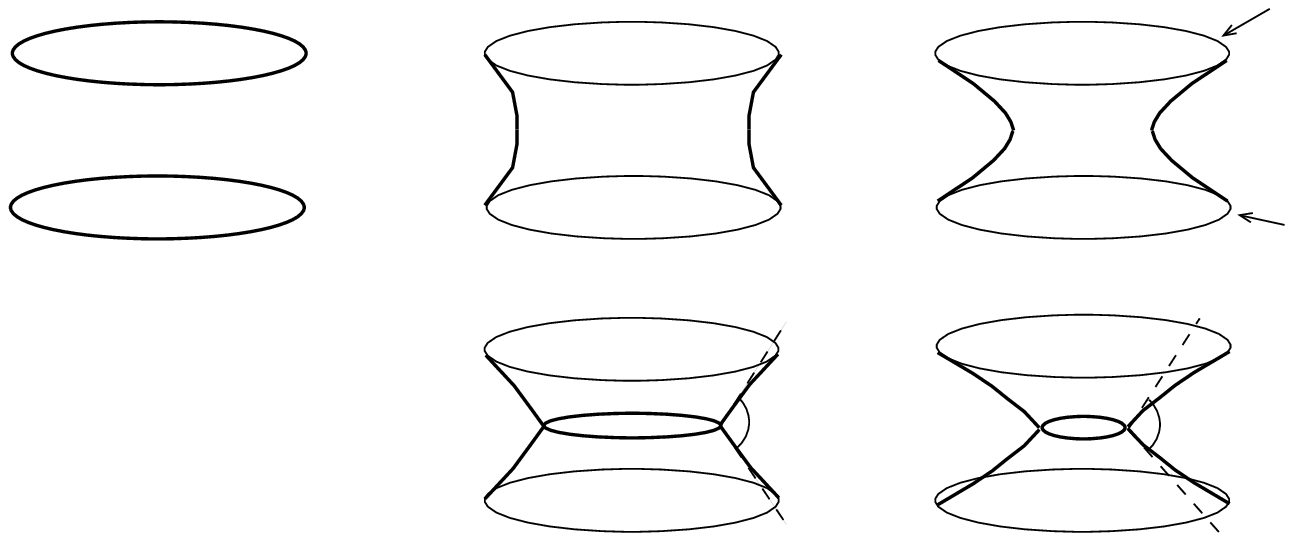\caption{{\small When $\Gamma$ consists of two parallel disks there are, in addition to the disconnected surface defined by two disks, four minimal surfaces, two of them singular, all composed by joining pieces of catenoids.}}\label{fig catenoids}
  \end{figure}
  We can construct generalized minimal surfaces on $\R^3\setminus\Gamma$ as multiplicity one varifolds $\var(N_i):=\var(N_i,1)$, associated to the rectifiable sets
  \begin{eqnarray*}
  &&N_1:=D_1\cup D_2\,,\qquad  N_2:=K_3\,,\qquad N_3:=K_4\,,
  \\
  &&N_4:=K_5\cup K_6\cup D_7\,,\quad N_5=K_8\cup K_9\cup D_{10}\,,
  \end{eqnarray*}
  depicted in Figure \ref{fig catenoids}, and referring to the following list of connected minimal surfaces:
  \begin{equation}\nonumber
    \begin{split}
      &\mbox{$D_1$ and $D_2$ are two disks spanned by $\Gamma_1$ and $\Gamma_2$ resp.}\,;
      \\
      &\mbox{$K_3$ and $K_4$ are the catenoids (one stable, the other unstable) spanned by $\Gamma$}\,;
      \\
      &\mbox{$K_5$ and $K_6$ are two catenoids meeting at a $2\pi/3$-angle along a circle $\Gamma_3$}
      \\
      &\qquad\qquad\mbox{lying on the midplane between $\Gamma_1$ and $\Gamma_2$, centered along the same axis}\,;
      \\
      &\mbox{$D_7$ is the disk spanned by $\Gamma_3$}\,;
      \\
      &\mbox{$K_8$ and $K_9$ are another pair of catenoids meeting at a $2\pi/3$-angle along a circle $\Gamma_4$}
      \\
      &\qquad\qquad\mbox{lying on the midplane between $\Gamma_1$ and $\Gamma_2$, centered along the same axis,}
      \\
      &\qquad\qquad\mbox{with the radius of $\Gamma_4$ smaller than the radius of $\Gamma_3$;}
      \\
      &\mbox{$D_{10}$ is the disk spanned by $\Gamma_4$}\,.
    \end{split}
  \end{equation}
  We claim that the $\var(N_i)$'s are generalized minimal surfaces. Since $N_4$ and $N_5$ are not smooth, we need to check carefully if they satisfy \eqref{minimal varifold}. By applying the tangential divergence theorem separately on the three minimal surfaces $K_5$, $K_6$ and $D_7$, we find that
  \[
  \int_{N_4}\,\Div^{N_4} X\,d\H^2=\int_{\Gamma_3} X\cdot\big(\nu_{\Gamma_3}^{K_5}+\nu_{\Gamma_3}^{K_6}+\nu_{\Gamma_3}^{D_7}\big)\,d\H^1\,.
  \]
  The sum of the above three co-normals is identically zero by the $2\pi/3$-angle condition imposed on $K_6$ and $K_7$, and so \eqref{minimal varifold} holds, thus showing that $N_4$ is minimal. The minimality of $N_5$ follows analogously. We also notice that every integer valued combination
  \begin{equation}
    \label{a1}
      V=\sum_{i=1}^5 q_i\,\var(N_i) \qquad \mbox{for some $q_i \in \N$}
  \end{equation}
  satisfies \eqref{minimal varifold}, and is thus a possible limit for a sequence $\{M_j\}_j$ satisfying \eqref{Mj assumption} with $\Gamma=\Gamma_1\cup\Gamma_2$. If such a limit arises with $\sum_i q_i\ge 2$, we speak of {\bf bubbling}. In fact, an additional subtlety lies in the fact that varifolds of the form
  \begin{equation}
    \label{a2}
      V=q_{1,1}\,\var(D_1)+q_{1,2}\,\var(D_2)+\sum_{i=2}^5 q_i\,\var(N_i)\qquad\mbox{with $q_{1,1}\ne q_{1,2}$}
  \end{equation}
  satisfy \eqref{minimal varifold}, and thus can arise as limits of almost-minimal surfaces (and indeed do so, see Example \ref{example delta meno infinito} below, if the mean curvature deficit is sufficiently weak). A limit like \eqref{a2} is qualitatively worse than a limit of the form \eqref{a1}, in the sense that $D_1$ and $D_2$ alone do not span the whole $\Gamma$, but just some of its connected components.}
\end{example}



\subsection{A geometric assumption: accessibility from infinity}\label{section infinity} Given $x\in\Gamma$, we say that $\Gamma$ is {\bf accessible from infinity at $x$} if there exist a unit vector $e$ and an angle $\theta\in[0,\pi)$ such that
\begin{equation}
  \label{accessible at x}
  \Gamma^{{\rm co}}\subset x+\Big\{z\in\R^{n+1}:z\cdot e\ge \frac{|z-(z\cdot e)e|}{\tan(\theta/2)}\Big\}\,,
\end{equation}
where $\Gamma^{{\rm co}}$ denotes the convex envelope of $\Gamma$. Notice that if \eqref{accessible at x} holds at a given $x$ then every minimal surface $N$ spanned by $\Gamma$ is automatically contained in the wedge centered at $x$ which appears on the right hand side of \eqref{accessible at x}.

\begin{definition}
  \label{definition totally}{\rm We say that $\Gamma$ is {\bf accessible from infinity} if, for each connected component $\Gamma_m$ of $\Gamma$, the set of points $x\in\Gamma_m$ such that $\Gamma$ is accessible from infinity at $x$ has {\it positive} $\H^{n-1}$-measure.}
\end{definition}

\begin{remark}
{\rm (i) Notice that $\Gamma$ does not need to be accessible at each of its points, we are just requiring that points of access have positive $\H^{n-1}$-measure inside each connected component of $\Gamma$; (ii) if $\Gamma$ is contained in the boundary of a uniformly convex set $K \subset \R^{n+1}$, then $\Gamma$ is accessible from infinity; the two conditions, on the other hand, are by no means equivalent, recall Figure \ref{fig accessible}.}
\end{remark}

Without $\Gamma$ being accessible we can easily construct examples where question \eqref{question} has negative answer, even if we intend almost-minimality in the strongest form defined by the uniform deficit.

\begin{example}[Negative answer to \eqref{question} and bubbling with uniform deficit] \label{current vs varifold}
  {\rm  Consider two concentric disks $S_1$ and $S_2$ contained inside a same plane, and bounded by circles $\Gamma_1$ and $\Gamma_2$, see
  \begin{figure}
  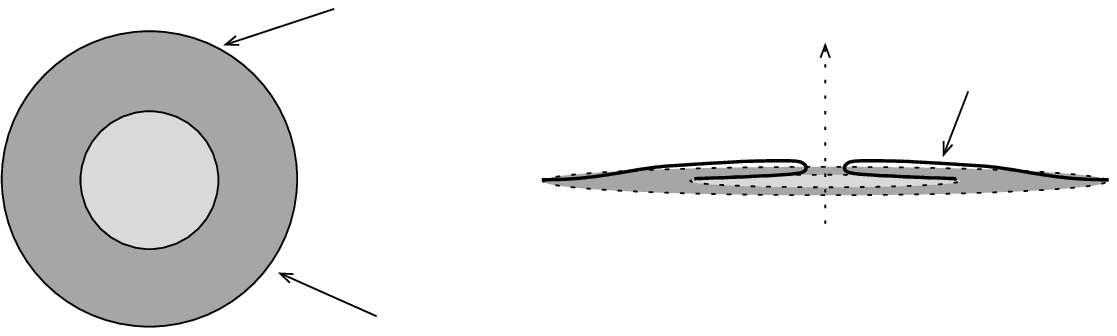\caption{{\small The construction described in Example \ref{current vs varifold}.}}\label{fig bubbling}
  \end{figure}
  Figure \ref{fig bubbling}. Set $\Gamma=\Gamma_1\cup\Gamma_2$, so that $N=S_1\setminus{\rm int}(S_2)$ is definitely a minimal surface spanned by $\Gamma$. Also, choose orientations on $S_1$, $S_2$ and $\Gamma$ in such a way that the spanning condition holds for the associated currents, that is $\partial \llbracket N \rrbracket = \partial (\llbracket S_1 \rrbracket - \llbracket S_2 \rrbracket) = \llbracket \Gamma \rrbracket$. We construct a sequence of surfaces $M_j$ by slightly bending $S_1$ and $S_2$ in the radial direction, and then connecting the two pieces with a catenoidal neck, see Figure \ref{fig bubbling}. Evidently, this can be arranged so that
\[
\pa \llbracket M_j \rrbracket = \llbracket \Gamma \rrbracket\,,\qquad \lim_{j\to\infty}\delta_{\infty}(M_j)=0\,,
\]
and the $M_j$'s converge to two copies of $S_2$ plus one copy of $N$, in the sense that
\begin{equation} \label{first bubbling and bad varifold}
\var(M_j) \to \var(N) +  2\var(S_2) = \var(N,1) + \var(S_2,2)\,.
\end{equation}
In particular:
\begin{equation}
  \label{first bubbling}
  \lim_{j\to\infty}\int_{M_j}\vphi\,d\H^n=\int_N\,\vphi\,d\H^n+2\,\int_{S_2}\,\vphi\,d\H^n\,,\qquad\forall\vphi\in C^0_c(\R^{n+1})\,.
\end{equation}
On the other hand, the currents $\llbracket M_j \rrbracket$ satisfy
\begin{equation} \label{limit in the sense of currents}
\llbracket M_j \rrbracket \to \llbracket N \rrbracket \qquad \mbox{in the sense of currents}\,,
\end{equation}
that is
\begin{equation} \label{limit in the sense of currents meaning}
\lim_{j \to \infty} \langle \llbracket M_j \rrbracket, \omega \rangle = \langle \llbracket N \rrbracket , \omega \rangle \qquad \forall\,\omega \in \D^{n}(\R^{n+1})\,,
\end{equation}
because the two copies of $S_2$ appearing in the limit come with opposite orientations, and hence the corresponding currents cancel out. For this simple boundary curve $\Gamma$, we thus have a negative answer to \eqref{question}: indeed, as shown by \eqref{first bubbling}, the limit of the $\{M_j\}_j$ cannot be described only in terms of minimal surfaces spanned by $\Gamma$ (which indeed is not spanning $S_2$).
In this example the bubbling phenomenon occurs, as part of the limit surface has multiplicity $2$. Observe also that $\Gamma$ is not accessible. Indeed, \eqref{accessible at x} cannot hold at any $x\in\Gamma_2$. Finally, the example can be easily generalized to the situation when $S_1$ and $S_2$ are two smooth, bounded, simply connected orientable minimal surfaces $S_1$ and $S_2$, spanned by curves $\Gamma_1$ and $\Gamma_2$, with $S_2\subset S_1$.}
\end{example}

\begin{example}[Bubbling under accessibility from infinity with very weak deficit]
  \label{example delta meno infinito}{\rm  As in Example \ref{example two rings part 1}, let $\Gamma$ consist of two parallel disks $\Gamma_1$ and $\Gamma_2$ with centers on a same axis, so that $\Gamma$ is accessible from infinity. We can give a negative answer to question \eqref{question} if a too weak notion of almost-minimality deficit is used, arguing along the following lines. Consider a catenoid $K$ spanned by $\Gamma$, and construct a sequence $M_j$ by slightly deforming $K$ outwards while keeping the boundary data at $\Gamma_2$, sharply turning around along $\Gamma_1$, going all the way towards the center of $\Gamma_1$, turning again downwards with a small catenoidal neck, and then almost filling $\Gamma_1$ with a disk; see
  \begin{figure}
    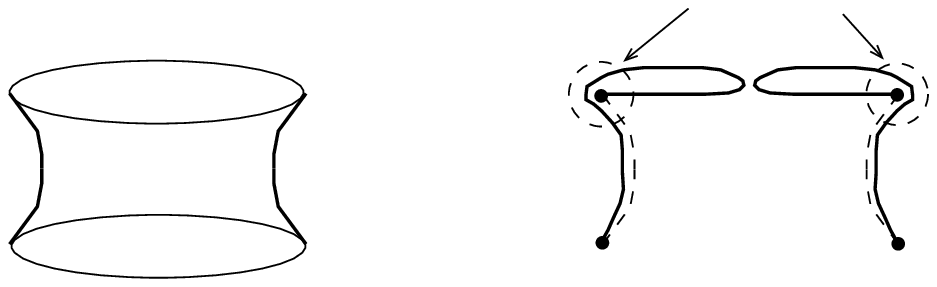\caption{\small{Bubbling is possible even when $\Gamma$ is accessible from infinity if a weak notion of deficit is used. Here $M_j$ is the surface of revolution obtained by rotating the one-dimensional profile on the right, $B_{\e_j}(\Gamma_1)$ denotes an $\e_j$-neighborhood of the circle $\Gamma_1$, and $M_j^*$ is the part of $M_j$ lying outside $B_{\e_j}(\Gamma_1)$. We take $\e_j$ such that $M_j$ intersects $\pa B_{\e_j}(\Gamma_1)$ in three circles, and so that the $H_{M_j}$ is uniformly small on $M_j\setminus M_j^*$. The limit surface counts one copy of $K$, and two copies of the disk filling $\Gamma_1$.}}\label{fig catenoids2}
  \end{figure}
  Figure \ref{fig catenoids2}. Denoting by $M_j^*$ the part of $M_j$ lying at distance at most $\e_j$ from $\Gamma_1$, by suitably selecting $\e_j\to 0$ as $j\to\infty$, we entail
  \[
  \pa M_j=\Gamma\,,\qquad \lim_{j\to\infty}\|H_{M_j}\|_{C^0(M_j\setminus M_j^*)}=0\,,\qquad \sup_{j\in\N}\|H_{M_j}\|_{L^1(M_j)}\le C\,.
  \]
  We claim that
  \[
  \lim_{j\to\infty}\de_{-\infty}(M_j)=0\,,\qquad \lim_{j\to\infty}\var(M_j)=2\,\var(D_1)+\var(K_3)\,,
  \]
  whereas
  \[
  \lim_{j\to\infty} \llbracket M_j \rrbracket = \llbracket K_3 \rrbracket\,.
  \]
  Thus the limits in the sense of varifolds and currents do not agree (we observe bubbling), while an almost-minimality deficit goes to zero (although this is indeed the weakest possible deficit in our scale). To show that $\de_{-\infty}(M_j)\to 0$, we fix a vector field $X$ compactly supported away from $\Gamma$ and with $|\nabla X|\le 1$. We fist notice that
  \[
  \Big|\int_{M_j^*}\Div^{M_j}X\Big|\le \H^2(M_j^*)\to 0\,.
  \]
  If $\Gamma_j^*$ is the component of the boundary of $M_j\setminus M_j^*$ that is not $\Gamma_2$, then by our choice of $\e_j$ we find
\begin{eqnarray*}
\Big|\int_{M_j\setminus M_j^*}\Div^{M_j}X\Big|&\le&\int_{M_j\setminus M_j^*}|X|\,|\HH_{M_j}|\,+\int_{\Gamma_j^*}\,|X|\,|\nu_{\Gamma_j^*}^{M_j}|
\\
&\le&  \diam(M_j)\,\H^2(M_j)\,\|H_{M_j}\|_{C^0(M_j\setminus M_j^*)}+\e_j\,\H^1(\Gamma_j^*)\,,
\end{eqnarray*}
where we have used $|\nabla X|\le 1$ and $X=0$ on $\Gamma$ to deduce: (i) that $|X|\le\e_j$ on $\Gamma_j^*$; and, (ii) that $|X|\le \diam(M_j)$ on $M_j$. Since $\H^1(\Gamma_j^*)\to 3\,\H^1(\Gamma_1)$ by construction, we have proved our claim.}
\end{example}

\subsection{Finiteness and regularity of the Plateau problem}\label{section finite plateau} The second main assumption we shall consider is that $\Gamma$ spans finitely many minimal surfaces. This is an idea that has to be formulated with great care, because of the singularities that minimal surfaces can exhibit.

Let $\Gamma$ be an $(n-1)$-dimensional compact smooth surface without boundary. As discussed in \S\, \ref{section convergence}, any varifold $V=\var(N,\theta)$ corresponding to a compact $\H^n$-rectifiable set $N$ in $\R^{n+1}$ and to a function $\theta\in L^1(\H^n\llcorner N;\N)$ such that
\begin{equation}
  \label{under}
  \int_N\,\theta\,\Div^N X\,d\H^n=0\qquad\forall X\in C^1_c(\R^{n+1}\setminus\Gamma;\R^{n+1})
\end{equation}
can arise as a possible limit of almost minimal surfaces. Possible limits $V$ have two other important properties: (i) As a consequence of \eqref{under}, the support of $V$ is bounded: indeed, an application of the monotonicity identity implies that $\spt\,V$ is contained in the convex hull of $\Gamma$, see \cite[Theorem 19.2]{SimonLN}; (ii) Given our assumptions on $M_j$, $V$ has bounded first variation, in the sense that
\[
\sup\Big\{\int_{N}\theta\,\Div^N\,X\,d\H^n:X\in C^1_c(\R^{n+1};\R^{n+1})\,,|X|\le 1\Big\}<\infty\,.
\]
In particular, by differentiation  of Radon measures, \eqref{under} is always extended to
\begin{equation}
  \label{under ext}
  \int_N\,\theta\,\Div^N X\,d\H^n=\int_{\R^{n+1}}X\cdot\nu\,d\mu_*\qquad\forall X\in C^1_c(\R^{n+1};\R^{n+1})\,,
\end{equation}
where $\mu_*$ is singular with respect to $\H^n\llcorner N$, and where $\nu$ is a Borel unit vector field.

Fully understanding the regularity of $\spt V$ when \eqref{under} holds is a major open problem in Geometric Measure Theory. What is known on this specific problem is the following. Define (for any compact set $N$) the sets of regular and singular points of $N$ as
 \begin{eqnarray*}
  {\rm Reg}(N)&:=&\Big\{x\in N \, : \,\mbox{$\exists\rho>0$ s.t. $N \cap B_{\rho}(x)$ is a smooth surface}
  \\
  &&\hspace{3cm}\mbox{with or without boundary in $B_{\rho}(x)$}\Big\}\,,
  \\
  \Sigma(N)&:=&N\setminus {\rm Reg}(N)\,.
\end{eqnarray*}
We further divide ${\rm Reg}(N)$ into ${\rm Reg}^{\circ}(N)$, the set of regular points of interior type (i.e., $N\cap B_{\rho}(x)$ is diffeomorphic to an $n$-dimensional disk), and into ${\rm Reg}^{b}(N)$, the regular points of boundary type. Now, let $V = \var(N,\theta)$ be such that \eqref{under} holds, and consider any open set $A$ such that $\theta$ is constant on $A\cap N$. Then Allard's regularity theorem \cite{Allard} shows that
\[
\H^n\Big((A\cap N)\Delta{\rm Reg}^\circ(N)\Big)=0\,.
\]
There is also a boundary regularity theorem \cite{Allardboundary}, showing the existence of $\e(n)>0$ such that if $\theta=1$ on $A\cap N$ and $\H^n(N\cap B_\rho(x))\le (1+\e(n))\,\om_n\rho^n/2$ for some $x\in A\cap N\cap\Gamma$, then $N\cap B_{\e(n)\rho}(x)$ is diffeomorphic to a half-disk.

The application of Allard's boundary regularity theorem can be quite deceptive. With reference to the notation of Example \ref{example two rings part 1}, it suffices to take $N=D_1\cup K_3$ with $\theta\equiv 1$ to construct an example of $V$ solving \eqref{under}, with $N\setminus\Gamma={\rm Reg}^\circ(N)$, and with $\Gamma_1=\Sigma(N)$. Notice also that a similar example holds even in the ``smoother'' case when the measure $\mu_*$ considered in the extension \eqref{under ext} of \eqref{under} actually agrees with $\H^{n-1}\llcorner\Gamma$, and when $\nu$ is $\H^{n-1}$-a.e. orthogonal to $\Gamma$; that is to say, when \eqref{under ext} takes the more geometric form
\begin{equation}
  \label{minimal varifold spanned by Gamma}
  \left\{
  \begin{split}
  &\nu(x)\in\SS^n\cap(T_x\Gamma)^\perp\qquad\mbox{for $\H^{n-1}$-a.e. $x\in\Gamma$}\,,
  \\
  &\int_N\,\theta\,\Div^N X\,d\H^n=\int_{\Gamma}\nu\cdot X\,d\H^{n-1}\qquad\forall X\in C^1_c(\R^{n+1};\R^{n+1})\,.
  \end{split}
  \right .
\end{equation}
Indeed, if the distance between the circles $\Gamma_1$ and $\Gamma_2$ in Example \ref{example two rings part 1} is such that $K_3$ meets with $D_1$ along $\Gamma_1$ at a $120$-degrees angle, then adding up the unit conormals of $D_1$ and $K_3$ on $\Gamma_1$ we obtain a unit vector
\[
\nu=\nu_{\Gamma_1}^{D_1}+\nu_{\Gamma_1}^{K_3}
\]
such that  \eqref{minimal varifold spanned by Gamma} holds, but still the boundary regularity theorem cannot be applied at any point of $\Gamma_1$, as $N\setminus\Gamma={\rm Reg}^\circ(N)$ and $\Gamma_1=\Sigma(N)$.

Summarizing, the analysis of almost-minimal surfaces spanned by $\Gamma$ unavoidably leads to consider minimal varifolds in $\R^{n+1}\setminus\Gamma$, but, in turn, these objects are only partially understood. Our compactness theorem will thus be conditional to assuming a rather precise structure for minimal varifolds in $\R^{n+1}\setminus\Gamma$. Namely, we shall require the possibility of decomposing them as linear combinations, with integer coefficients, of finitely many, unit density, connected pieces $N_i$ with unit conormals $\nu^{\rm co}_i$ along finite unions $\Gamma^{(i)}$ of connected components of $\Gamma$ (in particular, each piece $N_i$ may just be spanned by part of $\Gamma$); when removing its singular set and $\Gamma$, each piece $N_i$ is disconnected into at most finitely many smooth connected components. As explained in Proposition \ref{remark graphs} below, these assumptions hold in the fundamental case when $\Gamma$ is a graph over a convex surface.

\begin{definition}[Finiteness and regularity of minimal varifolds spanned by $\Gamma$]\label{definition finitely many}
  {\rm Let $\Gamma$ be a compact $(n-1)$-dimensional smooth surface without boundary in $\R^{n+1}$, and let $\{\Gamma_m\}_{m=1}^{M}$ denote the connected components of $\Gamma$. We say that $\Gamma$ {\bf spans finitely many minimal surfaces without boundary singularities} if there exists a {\bf finite family} $\{N_i\}_i$ of compact $\H^n$-rectifiable sets with the following properties:
  \begin{enumerate}
    \item[(i)] for each $i$, $N_i\setminus\Gamma$ is connected, and there exists a finite union $\Gamma^{(i)} = \bigcup_{m \in I^{(i)}} \Gamma_m$ of connected components of $\Gamma$ with
        \[
        N_i\cap\Gamma=\Gamma^{(i)}={\rm Reg}^b(N_i)\,,\qquad\Sigma(N_i)\cap\Gamma=\emptyset\,,
        \]
        and such that for some $\nu^{\rm co}_i:\Gamma^{(i)}\to\SS^n$ with $\nu^{\rm co}_i(x)\in (T_x\Gamma^{(i)})^\perp\cap T_xN_i$ it holds
        \[
        \int_{N_i}\,\Div^{N_i} X\,d\H^n=\int_{\Gamma^{(i)}}\nu_i^{\rm co}\cdot X\,d\H^{n-1}\qquad\forall X\in C^1_c(\R^{n+1};\R^{n+1})\,;
        \]
        moreover, ${\rm Reg}^{\circ}(N_i)$ has finitely many connected components $\{N_{i,\ell}\}_{\ell=1}^{L(i)}$ such that, for each $\ell$, $\cl(N_{i,\ell})\setminus\Sigma(N_i)$ is an orientable, smooth $n$-dimensional surface with boundary, whose boundary points are contained in $\Gamma^{(i)}$;

    \item[(ii)] if $V=\var(N,\theta)$ has bounded support, bounded first variation, and satisfies
    \begin{equation}
      \label{zero zero}
          \int_N\,\theta\,\Div^NX=0\qquad\forall X\in C^1_c(\R^{n+1}\setminus\Gamma;\R^{n+1})\,,
    \end{equation}
    then there exist $q_i\in\N$ such that
    \[
    V=\sum_i\,q_i\,\var(N_i)\,.
    \]
  \end{enumerate}
  }
\end{definition}

\begin{remark}\label{remark allard}
  {\rm By Allard regularity theorem and by property (i), for each $i$, $\var(N_i)$ is a minimal varifold in $\R^{n+1}\setminus\Gamma$ with constant unit density, and thus we have
  \[
  \H^n(\Sigma(N_i))=0\,.
  \]
  Notice that we are excluding the possibility that $\Sigma(N_i)$ intersects $\Gamma$: in other words, singularities are allowed, but not up to the boundary. In principle, this is the situation depicted in Figure \ref{fig catenoids}. It is not hard, however, to observe soap films with curves of singular points extending up to the wire frame, so we do not expect this assumption to be generic.}
\end{remark}

The problem of checking Definition \ref{definition finitely many} on some classes of examples, or even in simple explicit situations like the one described in Example \ref{example two rings part 1}, seems delicate. In the next proposition we address the case of graphs over convex boundaries.

\begin{proposition}\label{remark graphs}
  If $\Omega\subset\R^n \times \{0\}$ is a bounded connected open set with smooth and convex boundary, and if $\Gamma\subset\R^{n+1}$ is the graph of a smooth function $u$ over $\pa\Om$, then $\Gamma$ spans finitely many minimal surfaces in the sense of Definition \ref{definition finitely many}.
\end{proposition}

\begin{proof} Let us assume without loss of generality that $0\in\Omega$.
  Let $V=\var(N,\theta)$ be an integral varifold with bounded support satisfying \eqref{zero zero}. We first prove that $\spt V$ is contained in $\cl(\Om\times\R)$, where $\cl(A)$ denotes the closure of $A\subset\R^{n+1}$. Indeed let $H_\Om$ denote the mean curvature of $\pa\Om$ with respect to the outer unit normal to $\Om$. Consider the open cylinders $K(t)=t\,(\Om\times\R)$ for $t>1$. Since the support of $V$ is bounded, for $t$ large enough we have that $\spt V \Subset K(t)$. If $t_*=\inf\{ t: \spt V\Subset K(t)\}$, then $t_*<\infty$ and thus there exists $x = (x',x_{n+1})\in\spt\,V\cap\pa K(t_*)$ such that, in the ordering of $\nu_{\pa K(t_*)}(x) = \nu_{\Om}(x_*)$, $x_* := (x'/t_*,0)$, the smooth surface $\pa K(t_*)$ touches from above $\spt V$  locally at $x$. Let us assume that $x\in\R^{n+1}\setminus\Gamma$. Since $\pa K(t_*)$ is smooth, $\HH_{\pa K(t_*)}(x)\cdot\nu_{\pa K(t_*)}(x)=H_\Om(x_*)/t_*\ge0$, and $V$ is minimal in a neighborhood of $x$, by the strong maximum principle of Sch\"atzle \cite[Theorem 6.2]{schatzle} this is possible only if, locally at $x$, $\pa K(t_*)$ is contained in $\spt V$. Since $\spt V$ is anyway contained in $\cl(K(t_*))$, by a continuity argument, and by the connectedness of $\pa K(t_*)$, we obtain $\pa K(t_*)\subset\spt V$. This would be a contradiction, since $\spt V$ is bounded. Thus it must be that $x\in\Gamma$, i.e. $t_*=1$, and $\spt V\subset\cl(\Om\times\R)$.

  \medskip

  The classical area integrand theory (see, e.g. \cite[Chapter 1]{giustiDIRECTMETHODS}) implies the existence of a smooth extension of $u$ to the whole $\Om$, still denoted $u$, such that $G(u)=\{(z,u(z)):z\in\cl(\Om)\}$ satisfies
  \[
  \int_{G(u)}\Div^{G(u)}X\,d\H^n=\int_{\Gamma}X\cdot\nu_u\,d\H^{n-1}\qquad\forall X\in C^1_c(\R^{n+1};\R^{n+1})
  \]
  with $\nu_u(x)\in (T_x\Gamma)^\perp$ for every $x\in\Gamma$. Setting $N_1=G(u)$, properties (i) and (ii) in Definition \ref{definition finitely many} are clearly satisfied by $N_1$.

  \medskip

  We finally prove that $V=q\,\var(G(u))$ for some $q\in\N$. Since $\spt V$ is bounded and contained in the closure of $\Om\times\R$, we find that
  \[
  s^*=\inf\big\{s:x_{n+1}<s+u(z)\,\,\,\forall(z,x_{n+1})\in\spt V\big\}
  \]
  is finite. In particular, $s^*\,e_{n+1}+G(u)$ touches $\spt V$ from above in the ordering of $e_{n+1}$. If the touching point $x$ does not belong to $\Gamma$, then, again by Sch\"atzle's strong maximum principle we find that $s^*\,e_{n+1}+G(u)\subset\spt\,V$ with $s^*\ne 0$. But then $\spt V$ would have a contact point with $\pa\Om\times\R$ outside of $\Gamma$, where $V$ is minimal, and thus the strong maximum principle would imply $\pa\Om\times\R\subset\spt V$, once again against the boundedness of $\spt\,V$. The touching point $x$ of $s^*\,e_{n+1}+G(u)$ and $\spt V$ must thus lie on $\Gamma$, so that $s^*=0$, and $x_{n+1}\le u(z)$ whenever $(z,x_{n+1})\in\spt\,V$. An entirely similar argument shows that
  \[
  s_*=\sup\big\{s:x_{n+1}>s+u(z)\,\,\,\forall(z,x_{n+1})\in\spt V\big\}=0\,.
  \]
  so that we also have $x_{n+1}\ge u(z)$ whenever $(z,x_{n+1})\in\spt\,V$. We have thus proved that $G(u)=\spt\,V$. The constancy theorem for integral varifolds, \cite[Theorem 41.1]{SimonLN}, implies that $V=q\,\var(G(u))$ for a constant $q\in\N$.
\end{proof}

\section{The compactness theorem}\label{section proof of compactness thm}

We are finally ready to state and prove our main compactness theorem.

\begin{theorem}[Compactness theorem for almost-minimal surfaces]
  \label{thm main}
   Let $\Gamma$ be a smooth $(n-1)$-dimensional compact orientable manifold without boundary in $\R^{n+1}$, and let $\llbracket \Gamma \rrbracket$ be the $(n-1)$-current corresponding to the choice of an orientation $\tau_\Gamma$ on $\Gamma$. Assume that $\Gamma$ is accessible from infinity (see Definition \ref{definition totally}) and that $\Gamma$ spans finitely many minimal surfaces without boundary singularities (see Definition \ref{definition finitely many}).

   Let $\{M_j\}_j$ be a sequence of smooth $n$-dimensional surfaces, oriented by smooth unit normal vector fields $\nu_{M_j}$, and with smooth boundaries $\Gamma_j$ oriented in such a way that, if $\llbracket M_j \rrbracket=\llbracket M_j,\star\nu_{M_j},1 \rrbracket$, then
   \begin{eqnarray}\label{sequence of main theorem}
   \left\{
   \begin{split}
   &\partial \llbracket M_j \rrbracket = \llbracket \Gamma_j \rrbracket\,,
     \\
     &\sup_j\,\max\Big\{\H^n(M_j),\sup_{x\in M_j}|x|\Big\}<\infty\,,
     \\
     &\lim_{j\to\infty}\int_{M_j}|H_{M_j}|=0\,.
   \end{split}
   \right .
   \end{eqnarray}
   Assume that $\Gamma_j$ converges to $\Gamma$, in the sense that there exist Lipschitz maps $f_j \colon \Gamma \to \R^{n+1}$ with
   \begin{eqnarray}\label{boundary convergence}
   \left\{\begin{split}
   &f_j(\Gamma)=\Gamma_j\,,\qquad \sup_j\Lip(f_j)<\infty\,,
   \\
   &\lim_{j \to \infty}\|f_j - {\rm id}_{\Gamma}\|_{C^1(\Gamma)} =0\,.
   \end{split}\right .
   \end{eqnarray}
   Then, there exist an $\H^n$-rectifiable set $N$, and Borel vector fields $\nu_N:N\to\SS^n$ and $\nu:\Gamma\to\SS^n$ with
   \begin{eqnarray}
   &&\mbox{$\nu_{N}(x) \in (T_xN)^{\perp}$ for $\H^n$-a.e. $x\in N$}\,,
   \\
   &&\mbox{$\nu(x) \in (T_x\Gamma)^{\perp}$ for $\H^{n-1}$-a.e. $x\in \Gamma$}\,,
   \\
   \label{bdry in the sense of currents}
   &&\partial \llbracket N, \star \nu_N,1\rrbracket = \llbracket \Gamma \rrbracket\,,
   \\ \label{bdry in the sense of varifolds}
   &&\int_N\Div^N\,X=\int_\Gamma X\cdot\nu\,d\H^{n-1}\qquad\forall X\in C^1_c(\R^{n+1};\R^{n+1})\,,
   \end{eqnarray}
   and such that, up to extracting subsequences, $M_j\to N$ both in the sense of currents and in the sense of varifolds, i.e.
   \begin{equation}
     \label{tesi thm}
     \llbracket M_j \rrbracket\to \llbracket N, \star \nu_N,1 \rrbracket\,,\qquad \var(M_j)\to\var(N,1)\,.
   \end{equation}
   \end{theorem}

   \begin{remark} \label{rmk:classical_singularities}
  {\rm A point that we are not trying to formalize  here is that in situations like the one considered in Figure \ref{fig catenoids}, when $\Sigma(N)$, if present, is ``classical'', then one can actually prove that $\Sigma(N)=\emptyset$, thus concluding that smooth $M_j$'s cannot converge to minimal surfaces with singularities. To illustrate the idea, let $\Gamma_1$ and $\Gamma_2$ be the circles of Example \ref{example two rings part 1}, and fix orientations on $\Gamma_1$ and $\Gamma_2$ in order to define the associated currents $\llbracket \Gamma_1 \rrbracket$ and $\llbracket \Gamma_2 \rrbracket$. Suppose by contradiction that as a limit of a sequence $M_j$ of almost-minimal surfaces with $\partial \llbracket M_j \rrbracket = \llbracket \Gamma \rrbracket := \llbracket \Gamma_1 \rrbracket + \llbracket \Gamma_2 \rrbracket$ one obtains the singular minimal surface $N=K\cup K'\cup D$ obtained by gluing two catenoids $K$ and $K'$ to a disk $D$ along the boundary circle $\Sigma=\pa D$ with a $120$-degrees angle. Assign orientations to $K$, $K'$, and $D$ in such a way that
  \[
  \pa\llbracket K \rrbracket=\llbracket\Gamma_1\rrbracket + \sigma_1 \pa\llbracket D \rrbracket\,,
  \qquad
  \pa\llbracket K' \rrbracket= \llbracket\Gamma_2\rrbracket + \sigma_2 \pa\llbracket D \rrbracket\,,
  \qquad
  \sigma_{i} \in \{\pm 1\}\,.
  \]
  The limit current $T$ of the sequence $\llbracket M_j \rrbracket$ must then satisfy
  \[
  T=\a_1\,\llbracket K \rrbracket+\a_2\,\llbracket K' \rrbracket+\a_3\,\llbracket D \rrbracket
  \]
  for $\a_i\in\{\pm 1\}$, so that
  \[
  \pa T=\a_1\,\llbracket \Gamma_1 \rrbracket+\a_2\,\llbracket \Gamma_2 \rrbracket+(\a_1\s_1 + \a_2\s_2 + \a_3)\,\llbracket \pa D \rrbracket\,.
  \]
  Since $T$ is the limit of currents defined by the $M_j$'s, we also have
  \[
  \pa T= \llbracket \Gamma \rrbracket = \llbracket \Gamma_1 \rrbracket + \llbracket \Gamma_2 \rrbracket\,,
  \]
which implies $\a_1 = 1 = \a_2$ and
\[
\s_1 + \s_2 + \a_3 = 0\,,
\]
which is impossible, given $\s_1,\s_2,\a_3 \in \{-1,1\}$. A general argument along these lines can be repeated if assuming that a number of odd half-spaces meet along points in $\Sigma(N)$.}
  \end{remark}

   Before giving the proof of the theorem, we need to introduce some notation. Given an $n$-dimensional varifold $V$ on $\R^{n+1}$, that is, a Radon measure on $G^n=\R^{n+1}\times(\SS^n/\equiv)$ as described in section \ref{section almost minimal surfaces}, we denote by
   \[
   \de V(X)=\int_{G^n}\Div_TX(x)\,dV(x,T)
   \]
   the {\bf first variation} of $V$ along a vector field $X\in C^1_c(\R^{n+1};\R^{n+1})$. The {\bf weight of $V$} and the {\bf total first variation of $V$} are defined by
   \begin{eqnarray*}
     \|V\|(E)&:=&V(E\times(\SS^n/\equiv))
     \\
     \|\de V\|(A)&:=&\sup\Big\{\int_{G^n}\Div_TX(x)\,dV(x,T):X\in C^1_c(A;B_1(0))\Big\}\,,
   \end{eqnarray*}
   for every Borel set $E$ and open set $A$ in $\R^{n+1}$.
   Given an $n$-dimensional integer rectifiable current $T=\llbracket N, \star \nu_N, \alpha \rrbracket$, the {\bf mass of $T$} is the Radon measure
   \[
   \|T\|:=|\alpha|\,\H^n\llcorner N\,.
   \]
   We denote by
   \[
   V_T=\var(N,|\alpha|)
   \]
   the {\bf induced varifold of $T$}.

\begin{proof}[Proof of Theorem \ref{thm main}] {\it Step one}: We start by discussing the varifold limit of the $M_j$'s. By the area formula and by \eqref{boundary convergence} we have
\begin{equation} \label{uniform_bdry_mass}
\Ha^{n-1}(\Gamma_{j}) \leq C \Ha^{n-1}(\Gamma)\,.
\end{equation}
Setting $V_{j} := \var(M_j)$, by the tangential divergence theorem we have
\begin{equation} \label{first_variation}
\delta V_{j}(X) = \int_{M_j} \di^{M_j} X \, d\Ha^n = \int_{M_j} X \cdot {\bf H}_{M_j} \, d\Ha^n + \int_{\Gamma_j} X \cdot \nu_{\Gamma_j}^{M_j} \, d\Ha^{n-1}\,,
\end{equation}
for every $X \in C^{1}_{c}(\R^{n+1};\R^{n+1})$. In particular, \eqref{uniform_bdry_mass} and  $\delta_{1}(M_j) \to 0$ imply
\begin{equation} \label{uniform_first_var}
\limsup_{j\to\infty}\|\delta V_{j} \|(\R^{n+1})\leq\limsup_{j\to\infty}\left(\int_{M_j} \abs{{\bf H}_{M_j}} \, d\Ha^n + \Ha^{n-1}(\Gamma_j) \right) \leq C \Ha^{n-1}(\Gamma)\,,
\end{equation}
while at the same time $\|V_j\|(\R^{n+1})=\H^n(M_j)$. By \eqref{sequence of main theorem}, the supports of the $V_j$'s are contained in a fixed ball, and
\[
\sup_j\|V_j\|(\R^{n+1})\le C\,,
\]
thus by Allard's compactness theorem for integral varifolds (cf. \cite[Theorem 6.4]{Allard}, \cite[Theorem 42.7]{SimonLN}), there exists a not relabeled subsequence $V_{j}\to V$ as $j\to\infty$ for an integral varifold $V$. We notice that
\begin{equation}
  \label{minimal varifold for Gamma proof}
  \de V(X)=0\qquad\forall X \in C^{1}_{c}(\R^{n+1} \setminus \Gamma, \R^{n+1})\,.
\end{equation}
Indeed \eqref{boundary convergence} implies that if $\spt\,X \subset \R^{n+1} \setminus \Gamma$, then $\spt\,X \subset \R^{n+1} \setminus \Gamma_j$ for every $j$ large enough. Thus, $\de_1(M_j)\to 0$ and \eqref{first_variation} give
\[
\delta V(X) = \lim_{j \to \infty} \delta V_{j}(X) = \lim_{j \to \infty} \int_{M_{j}} X \cdot {\bf H}_{M_j} \, d\Ha^{n} = 0\,.
\]
for every $X \in C^{1}_{c}(\R^{n+1} \setminus \Gamma, \R^{n+1})$, as claimed.

\medskip

Since $\Gamma$ spans finitely many minimal surfaces without boundary singularities, \eqref{minimal varifold for Gamma proof} implies the existence of finitely many compact $\Ha^n$-rectifiable sets $\{N_i\}_{i=1}^k$ such that
\begin{equation} \label{eq:structure}
V = \sum_{i=1}^{k} q_i\, \var(N_i) \qquad \mbox{for some $q_i \in \N$}\,,
\end{equation}
where, for each $i$, $N_i\setminus\Gamma$ is connected, there exist a finite union $\Gamma^{(i)} = \bigcup_{m \in I^{(i)}} \Gamma_m$ of connected components of $\Gamma$ with
\begin{equation}\label{bar}
N_i\cap\Gamma=\Gamma^{(i)}={\rm Reg}^b(N_i)\,,\qquad\Sigma(N_i)\cap\Gamma=\emptyset\,,
\end{equation}
and a vector field $\nu^{\rm co}_i:\Gamma^{(i)}\to\SS^n$ with $\nu^{\rm co}_i(x)\in (T_x\Gamma^{(i)})^\perp\cap T_xN_i$ for $\Ha^{n-1}$-a.e. $x \in \Gamma^{(i)}$ such that
\begin{equation}
\label{proof comp tang Ni}
\int_{N_i}\,\Div^{N_i} X\,d\H^n=\int_{\Gamma^{(i)}}\nu_i^{\rm co}\cdot X\,d\H^{n-1}\qquad\forall X\in C^1_c(\R^{n+1};\R^{n+1})\,.
\end{equation}
Moreover, for each $i$,  ${\rm Reg}^{\circ}(N_i)$ has finitely many connected components $\{N_{i,\ell}\}_{\ell=1}^{L(i)}$ such that, for each $\ell$, $\cl(N_{i,\ell})\setminus\Sigma(N_i)$ is an orientable, smooth $n$-dimensional surface with boundary, whose boundary points are contained in $\Gamma^{(i)}$. As noticed in Remark \ref{remark allard}, \eqref{proof comp tang Ni} and Allard's regularity theorem imply
\begin{equation}
  \label{bar2}
  \H^n(\Sigma(N_i))=0\,.
\end{equation}
In particular, $N_i$ is $\H^n$-equivalent to ${\rm Reg}^{\circ}(N_i)$, so that can rewrite \eqref{eq:structure} as
\begin{equation} \label{varifold_structure}
V = \sum_{i=1}^{k} \sum_{\ell=1}^{L(i)} q_{i,\ell} \,\var(N_{i,\ell})\,,
\end{equation}
with $q_{i,\ell} = q_i$ for every $\ell = 1,\dots,L(i)$.

\bigskip

\noindent {\it Step two}: We now take the limit of the $M_j$'s in the sense of currents. Setting $T_{j} := \llbracket M_j \rrbracket $, by \eqref{uniform_bdry_mass}, $\sup_j\H^n(M_j)<\infty$, and by the Federer-Fleming compactness theorem \cite{federerfleming60}, see also \cite[Theorem 27.3]{SimonLN}), we have that $T_j\to T$ in the sense of currents, up to extracting subsequences, where $T$ is an integral current. The $C^1$-convergence of $\Gamma_j$ to $\Gamma$, $T_j\to T$, and $\partial T_j = \llbracket \Gamma_j \rrbracket$, are easily seen to imply $\partial T = \llbracket \Gamma \rrbracket$. Moreover, it is easily seen that, as Radon measures on $\R^{n+1}$,
\begin{equation} \label{measure_comparison}
\| T \| \leq \| V \|\,,
\end{equation}
since the mass of currents is lower semicontinuous, the weight of varifolds is continuous on sequences with bounded supports, and since $\|T_j\|=\H^n\llcorner M_j=\|V_j\|$. By \eqref{eq:structure},
\begin{equation}
  \label{pi}
  \spt(T)\,\, \subset\,\, \bigcup_{i=1}^{k} N_{i}\,\,\subset\,\,\Bigg[\Gamma\cup\bigcup_{i=1}^k\Sigma(N_i)\cup\bigcup_{i=1}^{k}\bigcup_{\ell=1}^{L(i)}N_{i,\ell}\Bigg]\,.
\end{equation}
Next we introduce the integral $n$-current $T_{i,\ell} := T \llcorner N_{i,\ell}$. Notice that $N_{i,\ell}$ is a smooth, connected $n$-dimensional surface, and that
\[
N_{i,\ell}\cap \spt\,\pa\,T_{i,\ell}=\emptyset
\]
since
\[
N_{i,\ell}\cap\spt\,\pa\,T_{i,\ell}\,\subset\,N_{i,\ell}\cap\spt\,\pa\,T = N_{i,\ell}\cap\Gamma\,\subset
{\rm Reg}^{\circ}(N_i)\cap\Gamma =\emptyset\,,
\]
thanks to \eqref{bar}. By the constancy theorem for integral currents (cf. \cite[Theorem 26.27]{SimonLN}), we find $\alpha_{i,\ell} \in \Z$ and realizations $\llbracket N_{i,\ell} \rrbracket$  of $N_{i,\ell}$ as multiplicity one integral currents such that
\[
T_{i,\ell} = \alpha_{i,\ell}\, \llbracket N_{i,\ell}\rrbracket\,.
\]
Since $\Ha^{n}(\Sigma(N_i)) = 0$, \eqref{pi} implies that
\begin{equation} \label{current_structure}
T = \sum_{i=1}^{k} \sum_{\ell=1}^{L(i)} \alpha_{i,\ell} \llbracket N_{i,\ell} \rrbracket\,.
\end{equation}
Applying the boundary operator in the sense of currents to \eqref{current_structure}, and recalling that $\pa T=\llbracket \Gamma \rrbracket$, we find that
\begin{equation} \label{current_structure b}
\llbracket\Gamma \rrbracket = \sum_{i=1}^{k} \sum_{\ell=1}^{L(i)} \alpha_{i,\ell} \,\pa\llbracket N_{i,\ell} \rrbracket\,.
\end{equation}
Recall that $\cl(N_{i,\ell})\setminus\Sigma(N_i)$ is a smooth surface with boundary, with boundary points contained in $\Gamma^{(i)}$. If $\Gamma_m$ is one of the components of $\Gamma^{(i)}$, then there is exactly one $\ell$ such that $\Gamma_m\cap{\rm Reg}^b[\cl(N_{i,\ell})\setminus\Sigma(N_i)]\ne \emptyset$, and, in correspondence to it,
\[
\Gamma_m\subset \cl(N_{i,\ell})\,.
\]
In particular, localizing \eqref{current_structure b} to $\Gamma_m$, and setting $\llbracket\Gamma_m \rrbracket=\llbracket\Gamma \rrbracket\llcorner\Gamma_m$, we have
\begin{equation} \label{current_structure b m}
\llbracket\Gamma_m \rrbracket =\sum_{i,\ell \, \colon \, \Gamma_m \subset \cl(N_{i,\ell})} \alpha_{i,\ell} \,\pa\llbracket N_{i,\ell} \rrbracket\llcorner\Gamma_m\,,
\end{equation}
and since $\Gamma_m$ itself is connected,
\[
\mbox{if $\Gamma_m\subset \cl(N_{i,\ell})$, then}\,\,\left\{
\begin{split}
\mbox{either}\qquad&\pa\llbracket N_{i,\ell} \rrbracket\llcorner\Gamma_m=\llbracket\Gamma_m \rrbracket\,,
\\
\mbox{or}\qquad&\pa\llbracket N_{i,\ell} \rrbracket\llcorner\Gamma_m=-\llbracket\Gamma_m \rrbracket\,.
\end{split}\right .
\]
In particular, for suitable $\sigma_{i,\ell}^m\in\{\pm 1\}$, we deduce from \eqref{current_structure b m} that
\begin{equation} \label{condition_coefficients_current}
\sum_{i,\ell \, \colon \, \Gamma_m \subset {\rm cl}(N_{i,\ell})} \sigma_{i,\ell}^{m}\, \alpha_{i,\ell} = 1\qquad \mbox{for every $m=1,\dots,M$}\,.
\end{equation}

\bigskip

\noindent {\it Step three}: We now link $T$ to $V$. Let $V_T$ denote the integral varifold associated with $T$, that is
\begin{equation}
  \label{VT}
  V_T=\sum_{i=1}^{k} \sum_{\ell=1}^{L(i)} |\alpha_{i,\ell}|\,\var(N_{i,\ell})\,.
\end{equation}
Noticing that
\[
V_{T_j}=\var(M_j)=V_j\,,
\]
and taking into account that $\Gamma_j$ is converging to $\Gamma$ in $C^1$, $V_j\to V$ as varifolds, and $T_j\to T$ as currents, we are allowed to apply White's theorem \cite[Theorem 1.2]{White_currents_varifolds} to deduce the existence of an integral varifold $W$ such that
\begin{equation} \label{white}
V = V_T + 2W\,.
\end{equation}
Therefore, it has to be
\[
2W = \sum_{i=1}^{k} \sum_{\ell=1}^{L(i)} (q_{i,\ell} - \abs{\alpha_{i,\ell}})\, \var(N_{i,\ell})\,,
\]
and the integrality condition of $W$ in turn yields that there exist $\beta_{i,\ell} \in \N$ such that
\begin{equation} \label{congruence_coefficients}
q_{i,\ell} - \abs{\alpha_{i,\ell}} = 2\beta_{i,\ell} \quad \mbox{ for every $i \in \{1,\dots,k\}$, $\ell \in \{1,\dots,L(i)\}$}\,.
\end{equation}
We now claim that, for every $m=1,...,M$,
\begin{equation} \label{odd_condition}
\sum_{i,\ell \, \colon \, \Gamma_m \subset {\rm cl}(N_{i,\ell})} q_{i,\ell}\,\, \mbox{ is an odd integer}\,.
\end{equation}
Indeed, using $a\equiv b\,\,{\rm mod}(2)$ as a shorthand for saying that $a$ and $b$ have the same parity, \eqref{congruence_coefficients} implies that
\[
\sum_{i,\ell \, \colon \, \Gamma_m \subset {\rm cl}(N_{i,\ell})} q_{i,\ell}\, \equiv \sum_{i,\ell \, \colon \, \Gamma_m \subset {\rm cl}(N_{i,\ell})} \abs{\alpha_{i,\ell}} \quad {\rm mod}(2)\,.
\]
At the same time
\begin{eqnarray*}
\sum_{i,\ell \, \colon \, \Gamma_m \subset {\rm cl}(N_{i,\ell})} (\abs{\alpha_{i,\ell}} - \sigma_{i,\ell}^{m}\, \alpha_{i,\ell})
&=& 2 \sum_{(\alpha_{i,\ell} > 0) \land (\sigma_{i,\ell}^{m} = -1)} \alpha_{i,\ell}\quad
- \quad 2 \sum_{(\alpha_{i,\ell} < 0) \land (\sigma_{i,\ell}^{m} = 1)} \alpha_{i,\ell} \,
\\
&&\equiv 0 \quad {\rm mod}(2)\,,
\end{eqnarray*}
so that, taking \eqref{condition_coefficients_current} into account, \eqref{odd_condition} holds.

\bigskip

\noindent {\it Step four}: We now exploit the assumption that $\Gamma$ is accessible from infinity,  to improve \eqref{odd_condition} and show that, for every $m=1,...,M$,
\begin{equation}
\sum_{i,\ell \, \colon \, \Gamma_m \subset {\rm cl}(N_{i,\ell})} q_{i,\ell} = 1\,;
\end{equation}
or, equivalently, that
\begin{equation} \label{one_condition}
\sum_{i \, \colon \, m \in I^{(i)}} q_{i}= 1 \,,\qquad\mbox{for every $m=1,...,M$}\,.
\end{equation}
Indeed, by \eqref{eq:structure} and by \eqref{proof comp tang Ni}, for every $X \in C^1_{c}(\R^{n+1}, \R^{n+1})$ we find
\begin{equation} \label{limit1}
\delta V(X)=\sum_{i=1}^{k} q_{i} \int_{N_{i}} \Div^{N_i} X\, d\H^{n} = \sum_{i=1}^{k} q_i \int_{\Gamma^{(i)}} X \cdot \nu^{\rm co}_i \, d\Ha^{n-1}\,,
\end{equation}
while, at the same time, $\int_{M_j}|H_{M_j}|\to 0$ and \eqref{first_variation} imply
\begin{equation} \label{limit2_partial1}
\de V(X)=\lim_{j \to \infty} \delta V_j(X) = \lim_{j \to \infty} \int_{\Gamma_j} X \cdot \nu_{\Gamma_j}^{M_j} \, d\Ha^{n-1}\,.
\end{equation}
By \eqref{boundary convergence}, $\Gamma_j = f_j(\Gamma)$ with $\|f_j-\Id\|_{C^1(\Gamma)}\to 0$ as $j\to\infty$, so that, by the area formula,
\begin{equation} \label{limit2_partial2}
\int_{\Gamma_j} X \cdot \nu_{\Gamma_j}^{M_j} \, d\Ha^{n-1} = \int_{\Gamma} (X \circ f_j) \cdot (\nu_{\Gamma_j}^{M_j} \circ f_j)\, {\bf J}^{\Gamma}f_j \, d\Ha^{n-1}\,,
\end{equation}
where the tangential Jacobian of $f_j$ along $\Gamma$ satisfies ${\bf J}^{\Gamma}f_j\to 1$ in $C^0(\Gamma)$, while of course $(X \circ f_j)\to X$ in $C^0(\Gamma)$. Considering that
\[
\| \nu_{\Gamma_j}^{M_j} \circ f_j \|_{L^\infty(\Gamma)} \leq 1
\]
for every $j$, by weak-$^*$ compactness, and up to extracting a subsequence, there exists a vector field $\sigma \in L^{\infty}(\Gamma,\R^{n+1})$ with $\| \sigma \|_{L^\infty(\Gamma)} \leq 1$ and a (not relabeled) subsequence such that
\[
\mbox{$\nu_{\Gamma_j}^{M_j} \circ f_j \weakstar \sigma$ in $L^{\infty}(\Gamma, \R^{n+1})$.}
\]
Combining \eqref{limit1}, \eqref{limit2_partial1}, \eqref{limit2_partial2} and this last information with
$(X \circ f_j)\,{\bf J}^{\Gamma}f_j\to X$ in $C^0(\Gamma)$, we find
\begin{equation} \label{limit2}
\sum_{i=1}^{k} q_i \int_{\Gamma^{(i)}} X \cdot \nu^{\rm co}_i = \int_{\Gamma} X \cdot \sigma \,,\qquad\forall X\in C^1_c(\R^{n+1};\R^{n+1})\,.
\end{equation}
If we set
\[
\hat{\Gamma}_{m} := \bigcup_{m' \neq m} \Gamma_{m'}\,,\qquad m=1,...,M\,,
\]
then \eqref{limit2} tested at $X \in C^{1}_{c}(\R^{n+1} \setminus \hat{\Gamma}_m;\R^{n+1})$ implies that for every such $X$,
\begin{equation} \label{limit3}
\sum_{i \, \colon \, m \in I^{(i)}} q_i \int_{\Gamma_m} X \cdot \nu^{\rm co}_i  = \int_{\Gamma_m} X \cdot \sigma \,,\qquad\mbox{for every $m=1,...,M$}\,.
\end{equation}
By arbitrariness of $X$,
\[
\sigma(x) = \sum_{i\, \colon \, m \in I^{(i)}} q_{i}\,  \nu^{\rm co}_i(x) \qquad \mbox{at $\Ha^{n-1}$-a.e. $x \in \Gamma_m$, and for every $m=1,...,M$}\,.
\]
This implies $\sigma(x)\in (T_x\Gamma)^\perp$ for $\H^{n-1}$-a.e. $x\in\Gamma$, as well that
\begin{equation} \label{key}
\bigg|\,\sum_{i\, \colon \, m \in I^{(i)}} q_{i}\,  \nu^{\rm co}_i(x)\bigg| \leq 1 \quad \mbox{for $\H^{n-1}$-a.e. $x \in \Gamma_m$}\,
\end{equation}
for every $m=1,...,M$.

\medskip

We are now ready to prove \eqref{one_condition}. Thanks to \eqref{odd_condition}, for every $m \in \{1,\dots,M\}$ we can find $p\in\N\cup\{0\}$ such that
\begin{equation}
  \label{2p1}
  \sum_{i\, \colon \, m \in I^{(i)}} q_{i} = 2p+1\,,
\end{equation}
and we want to show that it must always be $p=0$. Since $\Gamma$ is accessible from infinity, we can select $x_0 \in \Gamma_m$ such that \eqref{key} holds at $x=x_0$, and such that there exists a wedge $W$ (strictly contained in a half-space) with vertex at $x_0$ and containing $\Gamma^{{\rm co}}$. Up to rigid motions, we assume that $x_0=0$ and that
\[
W= \Big\{ (x_{1},x') \in \R^{n+1} \, \colon \, x_{1} \geq 0 \mbox{ and } \abs{x'} \leq x_{1} \tan(\phi) \Big\}\,,
\quad\mbox{for some $\phi\in[0,\pi/2)$}\,.
\]
The $n$-plane $\pi := e_{1}^{\perp} = \{x_1 = 0\}$ is then a supporting hyperplane to $\Gamma^{{\rm co}}$ at $x_0 = 0$. Furthermore, since $x_0 = 0$ is a point on $\Gamma_m \subset \Gamma$, the tangent space $T_{0}\Gamma$ is a linear subspace of $\pi$. We may assume that $T_{0} \Gamma = \{x_{1} = 0 = x_{n+1}\}$. Finally, by the classical convex hull property of minimal surfaces, we have $N_{i}\subset\Gamma^{{\rm co}} \subset W$ for every $i$.

\medskip

Now, for every $i$ such that $m \in I^{(i)}$, $\nu^{(i)} :=-\nu^{\rm co}_i(0)$ is a unit vector in the two-dimensional plane $(T_0\Gamma)^\perp = \{ x_{j} = 0 \mbox{ for } j=2,\dots,n\}$. In the coordinates $(x_{1}, x_{n+1})$, thanks to $N_i\subset W$, we find that $\nu^{(i)}$ points {\it inwards} $W$, and thus that
\[
\nu^{(i)} = \big( \cos \theta_i, \sin \theta_i \big)\qquad\mbox{for some $|\theta_i| \leq \phi$}\,.
\]
If $\{i_{1},\dots,i_{r(m)}\} \subset \{1,\dots,k\}$ is the set of indexes $i$ such that $m \in I^{(i)}$, we define the vectors $v_{1}, \dots, v_{2p+1}$ by setting
\begin{align*}
& v_{1} = v_{2} = \dots = v_{q_{i_1}} := \nu^{(i_1)}\,,\\
& v_{q_{i_1}+1} = v_{q_{i_1}+2} = \dots = v_{q_{i_{1}} + q_{i_{2}}} := \nu^{(i_2)}\,,\\
&\dots \\
& v_{2p+2 -q_{i_r}} = \dots = v_{2p+1} := \nu^{(i_r)}\,,
\end{align*}
so that, by \eqref{key} applied at $x=x_0=0$,
\[
-\sum_{i\, \colon \, m \in I^{(i)}} q_{i}\,  \nu^{\rm co}_i(0)
=\sum_{i\, \colon \, m \in I^{(i)}} q_{i}\, \nu^{(i)} = \sum_{h =1}^{2p+1} v_{h}
\]
has length $\le 1$. We conclude the proof by showing that, if $p\ge 1$, then
\begin{equation} \label{contradiction}
\bigg|\sum_{h=1}^{2p+1} v_{h}\bigg| > 1\,.
\end{equation}
A proof of \eqref{contradiction} is in \cite[Lemma 6.16]{DLR}. For the reader's convenience and for the sake of clarity, we {\it verbatim} repeat the argument used in \cite{DLR}. First, we order the vectors $v_{h}$ in such a way that $\theta_{1} \leq \theta_{2} \leq \dots \leq \theta_{2p+1}$. For every $j \leq p$, set $w_{j} := v_{j} + v_{2p+2-j}$. Using simple geometric considerations, one immediately sees that $w_{j}$ is a positive multiple of the vector
\[
\bigg( \cos \Big(\frac{\theta_{j} + \theta_{2p+2-j}}{2}\Big), \sin\Big( \frac{\theta_{j} + \theta_{2p+2-j}}{2}\Big) \bigg)\,.
\]
Since $\theta_{j} \leq \theta_{p+1} \leq \theta_{2p+2-j}$, the angle between the vectors $w_j$ and $v_{p+1}$ is
\[
\sphericalangle (w_{j}, v_{p+1}) = \bigg|\theta_{p+1} - \frac{\theta_{j} + \theta_{2p+2-j}}{2}\bigg| \leq \frac{\theta_{2p+2-j} - \theta_j}{2} \leq \phi < \frac{\pi}{2}\,,
\]
so that $w_{j} \cdot v_{p+1} > 0$. Then, we can use the Cauchy-Schwarz inequality to estimate
\[
\bigg|\sum_{h=1}^{2p+1} v_{h}\bigg| \geq \bigg( \sum_{h=1}^{2p+1} v_{h} \bigg) \cdot v_{p+1} = \sum_{j=1}^{p} w_{j} \cdot v_{p+1} + \abs{v_{p+1}} > 1\,.
\]
This proves \eqref{contradiction}.

\bigskip

\noindent {\it Step five}: We conclude the proof. By \eqref{one_condition}, for every $m=1,...,M$, adding up over those $i$ such that $m\in I^{(i)}$, we find $\sum q_i=1$. By exploiting this fact, we find that:
\begin{itemize}
\item $q_{i}\in\{0,1\}$ for every $i \in \{1,\dots,k\}$; in other words, it cannot be $q_i\ge2$;
\item if $q_i = 1$, then $q_{i'} = 0$ for any $i' \neq i$ such that $I^{(i)} \cap I^{(i')} \neq \emptyset$: hence, for every $m = 1,\dots,M$ there is one and only one $i = i_m$ with $m \in I^{(i_m)}$ and $q_{i_m} = 1$;
\item from \eqref{congruence_coefficients}: since $q_{i} \in \{0,1\}$ for every $i$, $\beta_{i,\ell} = 0$ for every $i \in \{1,\dots,k\}$ and $\ell \in \{1,\dots,L(i)\}$. Thus, if $q_{i} = 1$ then $\alpha_{i,\ell} = \pm 1$ for every $\ell$; if $q_{i} = 0$ then $\alpha_{i,\ell}=0$ for every $\ell$.
\end{itemize}
We can thus argue as follows. We set $m_1:=1$, and let $i_1$ be the only index in $\{1,\dots,k\}$ such that $1 \in I^{(i_1)}$ and $q_{i_1} = 1$. Next, let $m_2 := \min\{m \in 1,\dots,M \, \colon \, m \notin I^{(i_1)}\}$, and let $i_2$ be the corresponding index. Proceeding inductively, after a finite number $h$ of steps the set $\{ m \, \colon \, m \notin I^{(i_1)} \cup \dots \cup I^{(i_h)} \}$ will be empty. We finally set
\[
N := \bigcup_{r=1}^h N_{i_r}\,,
\]
and claim that $N$ satisfies the conclusions of the theorem. In order to verify \eqref{bdry in the sense of varifolds}, we define $\nu:\Gamma\to\SS^n$ by
\[
\nu(x)=\nu^{{\rm co}}_{i_m}(x)\in (T_x\Gamma)^\perp\qquad\mbox{if $x\in\Gamma_m$}\,,
\]
and use \eqref{proof comp tang Ni}. Noticing that $q_i=0$ if $i\neq i_r$ for every $r=1,...,h$, and $q_i=1$ otherwise, we see that
\[
V=\sum_{i=1}^kq_i\,\var(N_i)=\sum_{r=1}^h \var(N_{i_r})=\var(N)\,,
\]
so that $\var(M_j)\to\var(N)$, which is the second conclusion in \eqref{tesi thm}; and as for the first conclusion in \eqref{tesi thm}, $\|T\|\le\|V\|$ implies
\begin{equation} \label{current_final_check}
T = \sum_{r=1}^{h} \sum_{\ell=1}^{L(i_r)} \alpha_{i_r,\ell} \llbracket N_{i_r,\ell} \rrbracket\,,
\end{equation}
with $\alpha_{i_r,\ell} = \pm 1$. Taking into account that $\H^n(\Gamma \cup \bigcup_{r=1}^{h} \Sigma(N_{i_r})) =0$, we can now define a Borel orientation $\nu_{N} \colon N \to \Sf^n$ by setting $\left.\nu_{N}\right|_{N_{i_r,\ell}} := \alpha_{i_r,\ell}\, \nu_{N_{i_r,\ell}}$, where $\nu_{N_{i_r,\ell}}$ is the orientation defining the current $\llbracket N_{i_r,\ell} \rrbracket$. With this definition, equation \eqref{current_final_check} reads
\begin{equation} \label{victory}
T = \llbracket N, \star\nu_N,1 \rrbracket\,,
\end{equation}
which implies that $\llbracket M_j \rrbracket \to \llbracket N, \star \nu_N, 1 \rrbracket$ in the sense of currents. This completes the proof of \eqref{bdry in the sense of currents}, thus of the theorem.
\end{proof}

\section{Sharp decay estimates} \label{section decay rates} In this last section we refine the conclusions of Theorem \ref{thm main} with sharp quantitative estimates under the additional assumptions that: (i) the boundaries of the surfaces $M_j$ are fixed, i.e., we assume $\Gamma_j=\Gamma$; (ii) for some fixed $p>n$,
\begin{equation} \label{ass Lp bound mean curvature}
\sup_{j \in \Na} \int_{M_j} |H_{M_j}|^p  < \infty\,;
\end{equation}
and (iii) the limit minimal surface $N$ is classical, that is, $\Sigma(N)=\emptyset$. Under these assumptions, by combining Allard's regularity theorem \cite{Allard} and the implicit function theorem one can show the existence of smooth functions $u_j:N\to\R$ with $u_j=0$ on $\pa N$, and such that
\[
M_j=\Big\{x+u_j(x)\,\nu_{N}(x):x\in N\Big\}\,,\qquad\lim_{j\to\infty}\|u_j\|_{C^1(N)}=0\,.
\]
Assumption (i) is not really needed to parameterize $M_j$ over $N$. Indeed, one could obtain a global parametrization (possibly with non-trivial tangential components) as soon as $\Gamma_j$ converges to $\Gamma$ in, say, $C^{1,\a}$, see \cite{CiLeMaIC1,CiLeMaIC2}. Assumptions (ii) and (iii) are instead needed to have the quantitative regularity estimates of Allard, and the possibility to apply them to the $M_j$'s for proving the graphicality property. We omit the details of the argument leading to the existence of the functions $u_j$, since it has appeared many other times in the literature: for instance, see \cite{FigalliMaggiARMA,CiLeMaIC1,CiLeMaIC2,krummelmaggi,ciraolomaggi2017}.

\bigskip

We now collect some formulas concerning the geometry of almost-flat normal Lipschitz graphs over a smooth compact embedded orientable $n$-dimensional surface $N \subset \R^{n+1}$, and prove a basic $C^0$-estimate; see in particular \eqref{c0 estimate pre fuglede} below, whose proof should be compared with the argument of \cite[Section 4]{krummelmaggi}. Let us consider a Lipschitz function $u:N\to\R$ with $u=0$ on $\pa N$ and $\| u \|_{C^0(N)}+{\rm Lip}(u) \leq \e$ small enough depending on $N$. We set
\[
\psi_u(p) := p + u(p) \nu_N(p)\,, \qquad p \in N\,,
\]
and let $\Psi_u:=\psi_u(N)$. We also assume that $\Psi_u$ has a distributional mean curvature $H_{\Psi_u}\in L^1(\Psi_u)$, so that
\[
\int_{\Psi_u}\Div^{\Psi_u}X\,d\H^n=\int_{\Psi_u}X\cdot\nu_{\Psi_u}\,H_{\Psi_u}\,d\H^n\qquad\forall X\in C^1_c(\R^{n+1};\R^{n+1})
\]
where $\nu_{\Psi_u}$ is the normal to $\Psi_u$ induced by $\nu_N$ through $\psi_u$. By the area formula, it holds for every bounded Borel measurable function $g$ on $\Psi_u$ that
\begin{equation} \label{area of a graph}
\int_{\Psi_{u}}g\,d\H^n=\int_N\,(g\circ\psi_u)\,{\bf J}^N\psi_{u}\,.
\end{equation}
For every $\varphi \in C^{1}_{c}(N)$ and $t$ in a neighborhood of $0$, we consider the variation
\begin{eqnarray*}
\Psi_{u+t\vphi}&=&\Big\{p+u(p)\nu_N(p)+t\,\vphi(p)\,\nu_N(p):p\in N\Big\}
\\
&=&  \Big\{q+t\,(\vphi\,\nu_N)(\pi_N(q)):q\in\Psi_u\Big\}\,,
\end{eqnarray*}
where we denote by
\[
\pi_N:B_{\e_0}(N)\to N
\]
the smooth nearest point projection of the $\e_0$-neighborhood of $N$ onto $N$, and where of course we are assuming $\e<\e_0$. By the standard first variation formula for the area applied to $\Psi_u$ we find that
\begin{eqnarray}\nonumber
\frac{d}{dt}\Bigg|_{t=0}\H^n(\Psi_{u+t\vphi})&=&
\frac{d}{dt}\Bigg|_{t=0}\H^n\Big(\Big[\Id+t\,[(\vphi\nu_N)\circ\pi_N]\Big](\Psi_u)\Big)
\\\label{v20 1}
&=&\int_{\Psi_u}H_{\Psi_u}\nu_{\Psi_u}\cdot[(\vphi\,\nu_N)\circ\pi_N]\,.
\end{eqnarray}
Since $\pi_N$ restricted to $\Psi_u$ is the inverse of $\psi_u$, we have
\begin{eqnarray}
 \label{this formula}
  \int_N\left. \frac{d}{dt}\right|_{t=0}{\bf J}^N\psi_{u+t\vphi} &=&\frac{d}{dt}\Bigg|_{t=0}\H^n(\Psi_{u+t\vphi})
  \\\nonumber
  &=&\int_{N}\,\vphi\,\big(H_{\Psi_u}\circ\psi_u\big)\,(\nu_{\Psi_u}\circ\psi_u)\cdot\nu_N\,{\bf J}^N\psi_u\,.
\end{eqnarray}
We now want to compute $H_{\Psi_u}$ by using local coordinates. Let us cover $N$ by open sets $A\subset\R^{n+1}$ such that at every $p\in A\cap N$ we can define an orthonormal frame $\{\tau_i(p)\}_{i=1}^n$ for $T_pN$ with $\nabla_{\tau_i}\nu_N=\k_i\,\tau_i$, where $\k_1 \leq \k_2 \leq \dots\leq \k_n$ denote the principal curvatures of $N$. Setting $\pa_iu=\nabla_{\tau_i}u$ and $Du=(\pa_1u,...,\pa_nu)\in\R^n$, we find
\begin{equation}\label{jacobian locally}
{\bf J}^N\psi_{u}=G(p,u,D u)\quad\mbox{for $p\in A\cap N$}\,,
\end{equation}
where $G=G(p,z,\xi):(A\cap N)\times\R\times\R^n\to\R$ is given by
\begin{equation} \label{function G}
G(p,z,\xi)=\prod_{i=1}^n(1+\k_i\,z)\,\sqrt{1+\sum_{i=1}^n\Big(\frac{\xi_i}{1+\k_i z}\Big)^2}\,.
\end{equation}
Noticing that, on $A\cap N$,
\[
\nu_{\Psi_u}\circ\psi_u=\frac{\nu_N-\sum_{i=1}^n\pa^*_iu\,\tau_i}{\sqrt{1+|D^*u|^2}}\,,
\qquad\pa_i^*u=\frac{\pa_i u}{1+u\,\k_i\,,}\,,\qquad D^*u=(\pa_1^*u\,\dots,\pa_n^*u)\in\R^n\,,
\]
we find that
\[
(\nu_{\Psi_u}\circ\psi_u)\cdot\nu_N=\frac1{\sqrt{1+|D^*u|^2}}\qquad\mbox{on $A\cap N$}\,.
\]
By \eqref{jacobian locally} we also have, again on $A\cap N$,
\begin{equation} \label{this other formula}
\left. \frac{d}{dt}\right|_{t=0}{\bf J}^N\psi_{u+t\vphi}  = \frac{\pa G}{\pa z}(p,u,Du) \varphi + \sum_{i=1}^n\frac{\pa G}{\pa\xi^i}\,(p,u,D u) \pa_i\vphi \,.
\end{equation}
Thus, if we test \eqref{this formula} with $\vphi \in C^1_c(A\cap N)$, and then we integrate by parts, we obtain
\begin{equation}
  \label{rewrite}
  (H_{\Psi_u}\circ\psi_u)\,\frac{{\bf J}^N\psi_u}{\sqrt{1+|D^*u|^2}}=
\frac{\pa G}{\pa z}(p,u,D u) -\sum_{i=1}^n\pa_i\Big( \frac{\pa G}{\pa\xi^i}(p,u,D u)\Big)\,.
\end{equation}
To understand the structure of \eqref{rewrite}, we compute for the $\xi$-gradient of $G$,
\begin{equation} \label{xi-gradient of G}
\frac{\pa G}{\pa\xi_i}(p,z,\xi)=g(p,z,\xi)\frac{\xi_i}{(1+\k_i z)^2}\qquad g(p,z,\xi)=\frac{\prod_{i=1}^n(1+\k_i\,z)}{\sqrt{1+\sum_{i=1}^n\Big(\frac{\xi_i}{1+\k_i z}\Big)^2}}\,;
\end{equation}
and for the $z$-derivative of $G$,
\begin{equation} \label{z-derivative of G}
\frac{\pa G}{\pa z}(p,z,\xi)
=G(p,z,\xi)\sum_{j=1}^n\frac{\k_j}{1+z\,\k_j}\Bigg\{1-\frac{\xi_j^2}{(1+z\,\k_j)^2\Big(1+\sum_{i=1}^n\Big(\frac{\xi_i}{1+\k_i z}\Big)^2\Big)}\Bigg\}\,.
\end{equation}
By exploiting $\|u\|_{C^0(N)}+\Lip(u)<\e$, we thus find that for measurable functions $a_i$ and $b$ on $N\cap A$ with
\[
\|a_i-1\|_{L^\infty(N\cap A)}+\|b-1\|_{L^\infty(N\cap A)}\le C(N)\e
\]
we have
\[
\sum_{i=1}^n\pa_i\Big( \frac{\pa G}{\pa\xi^i}(p,u,D u)\Big)=\sum_{i=1}^n\pa_i(a_i\,\pa_iu)\,,
\qquad
\frac{\pa G}{\pa z}(p,u,Du)=b\,\sum_{i=1}^n\frac{\k_i}{1+\k_i\,u}\,.
\]
Using $0=H_N=\sum_{i=1}^n\k_i$ we get
\[
\frac{\pa G}{\pa z}(p,u,Du)=b\,\sum_{i=1}^n\Big(\frac{\k_i}{1+\k_i\,u}-\k_i\Big)
=-b\,u\,\,\sum_{i=1}^n\frac{\k_i^2}{1+\k_i\,u}=-c\,u
\]
where $c$ is a  non-negative, bounded measurable function defined on $A\cap N$. Overall \eqref{rewrite} can be rewritten as
\begin{equation}
  \label{rewritex}
  (H_{\Psi_u}\circ\psi_u)\,d=-c\,u-\sum_{i=1}^n\pa_i(a_i\,\pa_iu)\qquad\mbox{on $A\cap N$}
\end{equation}
where we have set for brevity $d={\bf J}^N\psi_u/\sqrt{1+|D^*u|^2}$, so that $\|d-1\|_{L^\infty(A\cap N)}\le C(N)\,\e$.

We finally formulate \eqref{rewritex} as an elliptic PDE on a domain of $\R^n$. To this end, up to decrease the size of $A$, we can introduce coordinates on $A\cap N$ by means of an embedding $F:U\subset\R^n\to \R^{n+1}$ of an open set $U$ with smooth boundary in the unit ball of $\R^n$ with $A\cap N=F(U)$ and $A\cap\bd(N)=F(\bd(U))$. We set $\s_i=(\pa F/\pa x^i)\circ F^{-1}$ so that $\{\s_i(p)\}_{i=1}^n$ is also a frame of $T_pN$ for each $p\in A\cap N$, and we have
\[
\tau_i=\sum_{k=1}^n t_i^k\,\s_k
\]
for suitable functions $t_k^i\in C^\infty(A\cap N)$. Setting $\bar f=f\circ F$ for functions $f$ defined on $F(U)=A\cap N$, we notice that if $\vphi\in C^1_c(A\cap N)$, then
\begin{eqnarray*}
  -\int_N\vphi\sum_{i=1}^n\pa_i(a_i\,\pa_iu)&=&\int_{A\cap N}\sum_{i=1}^n a_i\nabla_{\tau_i}u\,\nabla_{\tau_i}\vphi
  \\
  &=&\int_{A\cap N}\sum_{i,k,h=1}^n a_i\,t_i^k\,t_i^h\,\nabla_{\s_k}u\,\nabla_{\s_h}\vphi
  \\
  &=&\int_{U}\,JF\,\sum_{i,k,h=1}^n \bar{a}_i\,\bar{t}_i^k\,\bar{t}_i^h\,\frac{\pa\bar{u}}{\pa x^k}\,\frac{\pa\bar{\vphi}}{\pa x^h}
  \\
  &=&\int_{U}\,\Lambda\,\nabla\bar{u}\cdot\nabla \bar{\vphi}
\end{eqnarray*}
for the symmetric, bounded and uniformly elliptic tensor field
\[
\Lambda(x)=JF(x)\,\sum_{i=1}^n\,\bar{a}_i(x)\,v_i(x)\otimes v_i(x)\qquad v_i(x)=\sum_{k=1}^n\bar{t}^k_i(x)\,e_k\,.
\]
Notice that the ellipticity of $\Lambda$ relies on the facts that $\{v_i(x)\}_{i=1}^n$ is a basis of $T_x\R^n$, $F$ is an embedding, and $\{\tau_i(p)\}_{i=1}^n$ is a basis of $T_pN$. Thus we can understand \eqref{rewritex} as
\begin{equation}
  \label{rewrite2}
  (H_{\Psi_u}\circ\psi_u\circ F)\,\bar{d}=-\bar{c}\,\bar{u}-\Div(\Lambda\nabla\bar{u})\qquad\mbox{on $U$}\,,
\end{equation}
where $\bar{c}$ is non-negative and bounded, and $\|\bar{a}_i-1\|_{L^\infty(U)}+\|\bar{d}-1\|_{L^\infty(U)}\le C(N)\,\e$. By \cite[Theorem 8.17, Theorem 8.25]{GT}, we find that, for any ball $B_{4r}$ contained in the unit ball where $U$ is a smooth domain with boundary,
\[
\|\bar{u}\|_{C^0(B_r\cap U)}\le C(N,r,q)\,\Big\{\|\bar u\|_{L^2(B_{2r}\cap U)}+\|(H_{\Psi_u}\circ\psi_u\circ F)\|_{L^q(B_{2r}\cap U)}\Big\}\,,
\]
provided $q>n/2$ and assuming that the right-hand side is finite. Changing variables one more time, and exploiting a covering argument, we thus find
\begin{equation}
  \label{c0 estimate pre fuglede}
\|u\|_{C^0(N)}\le C(N,q)\,\Big\{\|u\|_{L^2(N)}+\|H_{\Psi_u}\circ\psi_u\|_{L^q(N)}\Big\}\,,\qquad\forall q>\frac{n}2\,.
\end{equation}
We now assume the strict stability of $N$, and use the formulas above in order to obtain a sharp quantitative estimates for {\it Lipschitz} graphs which only involves a very weak notion of deficit.

\begin{theorem}[Weak-deficit estimate on Lipschitz graphs]\label{thm graphs}
  Let $N$ be a smooth compact orientable $n$-dimensional surface in $\R^{n+1}$ with boundary, and let $u:N\to\R$ be a Lipschitz function with $u=0$ on $\pa N$. Consider the almost-mean curvature deficit
  \[
  \de(u):=\sup\Big\{\left.\int_N\frac{d}{dt}\right|_{t=0}{\bf J}^N\psi_{u+t\vphi}
  :\int_N|\nabla\vphi|^2\le 1\,,\vphi\in H^1_0(N)\Big\}\,.
  \]
  If $H_N \equiv 0$ and $N$ is strictly stable, in the sense that, for some $\l>0$,
  \[
  \int_{N} \abs{\nabla \varphi}^2 - \abs{A_{N}}^2 \varphi^2 \geq \lambda \int_{N} \varphi^2 \qquad \forall \varphi \in H^{1}_0(N)\,,
  \]
  then there exist positive constants $\e$ and $C$ depending on $N$ such that the condition
  \begin{equation}
    \label{epsilon}
      \|u\|_{C^0(N)}+\Lip(u)\le \e
  \end{equation}
  implies
  \begin{equation}
    \label{tesi grafici}
      0\le \H^n(\Psi_u)-\H^n(N)\le C(N)\,\de(u)^2\,,
  \end{equation}
  \begin{equation}
    \label{stima w12}
    \|u\|_{H^1(N)}\le C(N)\,\de(u)\,.
  \end{equation}
  In particular, if $\Psi_u$ has a distributional mean curvature $H_{\Psi_u}\in L^2(\Psi_u)$, then
  \begin{equation}\label{L2 estimates on graphs}
    0\le \H^n(\Psi_u)-\H^n(N)\le C(N)\,\int_{\Psi_u}\,|H_{\Psi_u}|^2\,,
  \end{equation}
  and if $H_{\Psi_u}\in L^q(\Psi_u)$ for some $q>n/2$, then
  \begin{equation}
    \label{stima C0 finale}
    \|u\|_{C^0(N)}\le C(N,q)\,\Big\{\|H_{\Psi_u}\|_{L^2(\Psi_u)}+\|H_{\Psi_u}\|_{L^q(\Psi_u)}\Big\}\,.
  \end{equation}
\end{theorem}

\begin{proof}[Proof of Theorem \ref{thm graphs}]
We first notice that if $H_{\Psi_u}\in L^2(\Psi_u)$, then
\begin{equation}\label{deficit comparison}
\delta(u)^2 \leq C(N) \int_{\Psi_u}\,|H_{\Psi_u}|^2 \,.
\end{equation}
Indeed, by \eqref{this formula}, by  H\"older inequality and by the Poincar\'e inequality on $N$, if $\vphi\in C^1_c(N)$ and $\pi_N$ is the normal projection over $N$, then
\begin{eqnarray*}
\bigg|\int_N\,\frac{d}{dt}\Big|_{t=0}J^N\psi_{u+t\vphi}\bigg|^2
&=&
\bigg|\int_{N}\,\vphi\,\big(H_{\Psi_u}\circ\psi_u\big)\,{\bf J}^N\psi_u\,(\nu_{\Psi_u}\circ\psi_u)\cdot\nu_N\bigg|^2
\\
&\leq&\,\int_N\vphi^2\,{\bf J}^N\psi_u\,\int_{N}\,\big(H_{\Psi_u}\circ\psi_u\big)^2\,{\bf J}^N\psi_u
\\
&\leq&\,\,\|{\bf J}^N\psi_u\|_{C^0(N)}\,\frac{\int_N|\nabla\vphi^2|}{c(N)}\,\int_{\Psi_u}\,|H_{\Psi_u}|^2
\end{eqnarray*}
so that \eqref{deficit comparison} immediately follows. In particular, \eqref{L2 estimates on graphs} is a consequence of \eqref{tesi grafici}, which we now prove. For the sake of clarity we shall first prove the theorem in the flat case when $\k_i \equiv 0$ for all $i = 1,\dots, n$, and thus $N$ is an open bounded set with smooth boundary in some $n$-plane of $\R^{n+1}$.

\medskip

\noindent {\it The flat case}: In this case we have
\[
\de(u)=\sup\Big\{\int_N\frac{\nabla u}{\sqrt{1+|\nabla u|^2}}\cdot\nabla\vphi:\int_N|\nabla\vphi|^2\le 1\,,\vphi\in H^1_0(N)\Big\}\,.
\]
Setting
  \[
  f(t):=\H^n(\Psi_{t\,u})=\int_N\,\sqrt{1+|t\,\nabla u|^2}\,,
  \]
we find
  \[
  \H^n(\Psi_u)-\H^n(N)=\int_0^1f'(t)\,dt=\int_0^1\,t\,dt\int_N\,\frac{\nabla u}{\sqrt{1+|t\nabla u|^2}}\cdot\nabla u=a+b\,,
  \]
where, by definition,
  \begin{eqnarray*}
    a&:=&\int_0^1\,t\,dt\int_N\,\frac{\nabla u}{\sqrt{1+|\nabla u|^2}}\cdot\nabla u\,,
    \\
    b&:=&\int_0^1\,t\,dt\int_N\,\Big\{\frac{1}{\sqrt{1+|t\nabla u|^2}}-\frac{1}{\sqrt{1+|\nabla u|^2}}\Big\}\,|\nabla u|^2\,.
  \end{eqnarray*}
  Clearly we have
  \[
  0\le a\le\frac12\,\de(u)\,\|\nabla u\|_{L^2(N)}\,.
  \]
  Now consider the function $g(\xi)=(1+|\xi|^2)^{-1/2}$, so that
  \[
  |\nabla g(\xi)|=\frac{\abs{\xi}}{(1+|\xi|^2)^{3/2}}\le |\xi|\,,
  \]
  and thus, for every $t\in[0,1]$ and $|\xi|\le1$,
  \[
  |g(t\xi) - g(\xi)| = \Big|\int_{t}^{1} \nabla g(s\xi) \cdot \xi \, ds\Big| \leq \frac{\abs{\xi}^2}{2}\,.
  \]
  This implies that
  \[
  |b|\le \int_0^1\,t\,dt\,\int_N\,|\nabla u|^4\le \Lip(u)^2\,\int_N|\nabla u|^2\,.
  \]
  So far we have proved
  \[
  0\le\H^n(\Psi_u)-\H^n(N)\le \de(u)\,\|\nabla u\|_{L^2(N)}+\Lip(u)^2\,\int_N|\nabla u|^2\,.
  \]
  We now notice that
  \[
  \sqrt{1+|\nabla u|^2}-1=\frac{|\nabla u|^2}{1+\sqrt{1+|\nabla u|^2}}\ge\frac{|\nabla u|^2}3\,,
  \]
  since $\Lip(u)\le\e\le 1$ (this last inequality expresses the strict stability of $N$). In particular,
  \[
  \Big(\frac13-\e^2\Big)\int_N|\nabla u|^2\le\de(u)\,\|\nabla u\|_{L^2(N)}\,,
  \]
  that is
  \[
  \Big(\frac13-\e^2\Big)^2\int_N|\nabla u|^2\le\de(u)^2\,.
  \]
  Since we also have $\sqrt{1+|\xi|^2}-1\le|\xi|^2$ we finally conclude that
  \[
  \Big(\frac13-\e^2\Big)\,\Big\{\H^n(\Psi_u)-\H^n(N)\Big\}\le\de(u)^2\,.
  \]
  This proves the theorem in the flat case.

  \medskip

  \noindent {\it The general case}: Our starting point is again the obvious remark that
  \begin{eqnarray}
    \nonumber
      \H^n(\Psi_u)-\H^n(N)&=&\int_0^1\frac{d}{dt}\H^n(\Psi_{tu})\,dt
      =\int_0^1\,\frac{d}{ds}\Bigg|_{s=0}\H^n(\Psi_{t\,u+s\,u})\,dt
      \\   \label{obvious}
      &=&\int_0^1dt\,\int_N\frac{d}{ds}\Bigg|_{s=0}{\bf J}^N\psi_{t\,u+s\,u}\,.
  \end{eqnarray}
  We recall that, by \eqref{this other formula}, with an open set $A$ as described at the beginning of the section and for every $\vphi\in C^1_c(N)$, it holds
  \[
  \left. \frac{d}{ds}\right|_{s=0}{\bf J}^N\psi_{u+s\vphi}  = \frac{\pa G}{\pa z}(p,u,Du) \varphi + \sum_{i=1}^n\frac{\pa G}{\pa\xi^i}\,(p,u,D u) \pa_i\vphi \qquad\mbox{on $A\cap N$}\,.
  \]
  With reference to \eqref{obvious} we thus have, on $A\cap N$,
  \begin{equation}
    \label{R plus tS}
      \frac{d}{ds}\Bigg|_{s=0}{\bf J}^N\psi_{t\,u+s\,u}
  =\frac{\pa G}{\pa z}(p,t\,u,t\,Du) u + \sum_{i=1}^n\frac{\pa G}{\pa\xi^i}\,(p,t\,u,t\,D u) \pa_iu=R+t\,S
  \end{equation}
  where we have set
  \begin{eqnarray*}
  R&=&\Big\{\frac{\pa G}{\pa z}(p,t\,u,t\,D u)-t\,\frac{\pa G}{\pa z}(p,u,D u)\Big\}\,u
  \\
  &&+\sum_{i=1}^n\Big\{\frac{\pa G}{\pa \xi^i}(p,t\,u,t\,D u)-t\,\frac{\pa G}{\pa \xi^i}(p,u,D u)\Big\}\cdot \pa_i u\,,
  \end{eqnarray*}
  and
  \begin{equation}
    \label{identity for S}
      S=\frac{\pa G}{\pa z}(p,u,D u)\,u+\sum_{i=1}^n\frac{\pa G}{\pa \xi^i}(p,u,D u)\cdot \pa_i u=
  \frac{d}{ds}\Bigg|_{s=0}{\bf J}^N\psi_{u+s\,u}\,.
  \end{equation}
  We now claim that
  \begin{equation}
    \label{stima per R}
    |R|\le C(N)\,\e\,(u^2+|\nabla u|^2)\qquad\mbox{on $A\cap N$}\,.
  \end{equation}
  Notice that the right-hand side of \eqref{stima per R} is independent from the embedding $F$, indeed $|\nabla u|^2=\sum_{i=1}^n(\pa_iu)^2=|Du|^2$ is the squared norm of the tangential gradient of $u$. Analogously for the right-hand side of \eqref{identity for S}, so that we can combine \eqref{R plus tS}, \eqref{identity for S} and \eqref{stima per R} to obtain
  \begin{equation}\label{stima finale L2}
  \frac{d}{ds}\Bigg|_{s=0}{\bf J}^N\psi_{t\,u+s\,u}
  \le   t\,\frac{d}{ds}\Bigg|_{s=0}{\bf J}^N\psi_{u+s\,u}+\e\,C(N)\,(u^2+|\nabla u|^2)\qquad\mbox{on $N$}\,,
  \end{equation}
  from which we will easily conclude the proof; but we first prove \eqref{stima per R}. We start by using \eqref{xi-gradient of G} to compute
\begin{eqnarray*}
  &&\sum_{i=1}^n\Big\{\frac{\pa G}{\pa\xi^i}(p,t\,u,t\,D u)-t\,\frac{\pa G}{\pa\xi^i}(p,u,D u)\Big\}\pa_iu
  \\
  &=&
  \big(g(p,t\,u,t\,Du)-g(p,u,D u)\big)\sum_{i=1}^n\,\frac{t(\pa_iu)^2}{(1+t\,\k_i u)^2}
  \\
  &&+g(p,u,D u)\sum_{i=1}^n\Big(\frac{t}{(1+t\k_i u)^2}-\frac{t}{(1+\k_iu)^2}\Big)\,(\pa_iu)^2\,,
\end{eqnarray*}
where
\[
\big|g(p,tu,t\,D u)-g(p,u,D u)\big|\le C(N)\,(1-t)\,\big(|u|+|\nabla u|\big)\,,
\]
and where, recalling $H_N=0$ and $0\le t\le 1$, we get
\[
\Big|\sum_{i=1}^n\frac{t}{(1+t\k_i u)^2}-\frac{t}{(1+\k_iu)^2}\Big|\le C(N)\,u^2\,.
\]
Hence we obtain
\begin{eqnarray}\nonumber
\Big|\sum_{i=1}^n\Big\{\frac{\pa G}{\pa\xi^i}(p,t\,u,t\,D u)-t\,\frac{\pa G}{\pa\xi^i}(p,u,D u)\Big\}\pa_iu\Big|
&\le& C(N)\,|\nabla u|^2\,(|u|+|\nabla u|)
\\\label{addio1}
&\le& C(N)\,\e\,|\nabla u|^2\,.
\end{eqnarray}
In order to bound
\[
\Big\{\frac{\pa G}{\pa z}(p,t\,u,t\,D u)-t\,\frac{\pa G}{\pa z}(p,u,D u)\Big\}\,u\,,
\]
we set $h(p,z,\xi)=(\pa G/\pa z)(p,z,\xi)$ and use \eqref{z-derivative of G} to find
\[
h(p,0,0)=G(p,0,0)\,\sum_{j=1}^n\k_j=H_N=0\,,
\]
and thus
\begin{eqnarray*}
  h(p,t\,u,t\,D u)-t\,h(p,u,D u)
  &=&\frac{\pa h}{\pa z}(p,0,0)\,(tu)+t\,\sum_{i=1}^n\frac{\pa h}{\pa \xi^i}(p,0,0)\,\pa_iu
  \\
  &&-t\,\frac{\pa h}{\pa z}(p,0,0)u-t\,\sum_{i=1}^n\frac{\pa h}{\pa \xi^i} (p,0,0)\,\pa_iu
  \\
  &&+{\rm O}(u^2+|\nabla u|^2)
  ={\rm O}(u^2+|\nabla u|^2)\,.
\end{eqnarray*}
Combined with $\|u\|_{C^0(N)}\le\e$ and with \eqref{addio1}, this estimate proves \eqref{stima per R}.

Having proved \eqref{stima finale L2}, we now complete the proof as follows. By \eqref{obvious},
\begin{eqnarray}\nonumber
  \H^n(\Psi_u)-\H^n(N)&\le&
  \int_0^1\,dt\,\int_N\,t\,\frac{d}{ds}\Bigg|_{s=0}{\bf J}^N\psi_{u+s\,u}+C(N)\,\e\,(u^2+|\nabla u|^2)
  \\\nonumber
  &\le&\de(u)\,\|\nabla u\|_{L^2(N)}+C(N)\,\e\,\int_N\,u^2+|\nabla u|^2
  \\\label{ciao}
  &\le&\de(u)\,\|\nabla u\|_{L^2(N)}+C(N)\,\e\,\int_N\,|\nabla u|^2\,,
  \end{eqnarray}
  where we have used \eqref{stima finale L2}, the definition of $\de(u)$ and the Poincar\'e inequality on $N$. By the strict stability of $N$ and thanks to a classical computation (see, e.g., \cite[Lemma 3.2]{dephilippismaggi}) we find
\[
\H^n(\Psi_u)-\H^n(N)\ge\frac1{C(N)}\,\int_N|\nabla u|^2\,,
\]
and we thus conclude that, if $\e$ is suitably small,
\[
\int_N|\nabla u|^2\le C(N)\de(u)\,\|\nabla u\|_{L^2(N)}\,,
\]
that is
\begin{equation} \label{gradient bound mc deficit}
\int_N|\nabla u|^2\le C(N)\,\de(u)^2\,,
\end{equation}
and \eqref{stima w12} follows by combining \eqref{gradient bound mc deficit} with the Poincar\'e inequality. Combining \eqref{ciao} with \eqref{gradient bound mc deficit} we prove \eqref{tesi grafici}. To prove the $C^0$ estimate we combine \eqref{c0 estimate pre fuglede}, i.e.
\begin{eqnarray*}
  \|u\|_{C^0(N)}\le C(N,q)\,\Big\{\|u\|_{L^2(N)}+\|H_{\Psi_u}\circ\psi_u\|_{L^q(N)}\Big\}\,,
\end{eqnarray*}
with \eqref{stima w12}, \eqref{deficit comparison} and the area formula. This completes the proof of the theorem.
\end{proof}

\bibliographystyle{is-alpha}
\bibliography{references}

\end{document}